\documentclass{amsart}

\usepackage{amsmath,amsthm,amsfonts, amssymb}
\usepackage[usenames,dvipsnames,svgnames,table]{xcolor}
\usepackage{url}
\usepackage[all]{xy} 
\usepackage{nag} 
\usepackage{bbm}
\usepackage{graphicx}

\usepackage{algorithmicx}
\usepackage{algorithm}
\usepackage{algpseudocode}

\newtheorem{alg}{Algorithm}
\renewenvironment{algorithm}[1]
{\vspace{10pt}
\begin{alg}\rm
{#1}
\vspace{5pt}
\hrule
\vspace{5pt}}
{\vspace{5pt}
\hrule
\end{alg}
\vspace{10pt}}

\algnewcommand\algorithmicinitialsetup{\textbf{Initial Setup:}}
\algnewcommand\InitialSetup{\item[\algorithmicinitialsetup]}
\algnewcommand\algorithmicpublickey{\textbf{Public Key:}}
\algnewcommand\PublicKey{\item[\algorithmicpublickey]}
\algnewcommand\algorithmicprivatekey{\textbf{Private Key:}}
\algnewcommand\PrivateKey{\item[\algorithmicprivatekey]}
\algnewcommand\algorithmicencoding{\textbf{Encryption:}}
\algnewcommand\Encoding{\item[\algorithmicencoding]}
\algnewcommand\algorithmicdecoding{\textbf{Decryption:}}
\algnewcommand\Decoding{\item[\algorithmicdecoding]}

\newtheorem{thm}{Theorem}[section]
\newtheorem{theorem}[thm]{Theorem}
\newtheorem{proposition}[thm]{Proposition}
\newtheorem{lemma}[thm]{Lemma}
\newtheorem{corollary}[thm]{Corollary}

\theoremstyle{remark}

\newtheorem{remark}[thm]{Remark}
\newtheorem{claim}[thm]{Claim}

\newenvironment{thma}[1][Theorem]{\begin{trivlist}\item[\hskip \labelsep {\bfseries #1}]}{\end{trivlist}}

\newcommand{\CA}{{\mathcal A}}

\newcommand{\CC}{{\mathcal C}}
\newcommand{\CF}{{\mathcal F}}
\newcommand{\CM}{{\mathcal M}}

\newcommand{\CL}{{\mathcal L}}
\newcommand{\CP}{{\mathcal P}}
\newcommand{\CS}{{\mathcal S}}
\newcommand{\CO}{{\mathcal O}}
\newcommand{\ovw}{{\overline{w}}}
\newcommand{\tdw}{{\widehat{w}}}
\DeclareMathOperator{\supp}{supp}
\DeclareMathOperator{\ncl}{ncl}
\DeclareMathOperator{\REW}{RandomEqualWord}

\def\MN{{\mathbb{N}}}

\def\MR{{\mathbb{R}}}
\def\tO{{\tilde O}}
\newcommand{\rb}[1]{{\left( #1 \right)}}
\DeclareMathOperator{\WSP}{WSP}
\DeclareMathOperator{\ESP}{ESP}
\DeclareMathOperator{\CSP}{CSP}
\DeclareMathOperator{\MSP}{MSP}

\newcommand{\PS}{\Omega}
\newcommand{\Prob}[1]{\Pr \left\{#1\right\}}
\newcommand{\CProb}[2]{\Pr \left\{#1 \mid #2\right\}}
\newcommand{\ME}[1]{{\mathbb{E}} \left( #1 \right)}
\newcommand{\Indicator}[1]{\mathbbm{1} \left(\, #1 \,\right)}
\newcommand{\URandom}[1]{\textbf{U}(#1)}
\newcommand{\DRandom}[1]{\mathcal{P}(#1)}
\newcommand{\LDRandom}[1]{\mathcal{LP}(#1)}
\DeclareMathOperator{\ExpDSign}{EXP}
\newcommand{\ExpD}[1]{\ExpDSign(#1)}

\newcommand{\gp}[1]{{\left\langle #1 \right\rangle}}
\newcommand{\gpr}[2]{{\left( #1 ; #2 \right)}}
\newcommand{\group}[2]{{\left\langle #1 \mid #2 \right\rangle}}
\newcommand{\Reduced}[1]{\overline{#1}}
\newcommand{\CReduced}[1]{\widehat{#1}}
\newcommand{\FG}{F}
\newcommand{\FreeGroup}[1]{\FG(#1)}
\newcommand{\FMnorm}[1]{|#1|}
\newcommand{\Fnorm}[1]{|#1|}
\newcommand{\conj}[1]{\sim_{#1}}

\newcommand{\Set}[2]{\left\{\, #1 \;\middle|\; #2 \,\right\}}

\newcommand{\EP}{\mathbb{R}^2}

\newcommand{\ElementaryIdentities}{\mathcal{I}}
\newcommand{\DiagElementaryIdentities}{\mathcal{I}_D}
\newcommand{\treeHeight}[1]{h(#1)}
\newcommand{\weight}[2][]{\gamma^{#1(#2)}}
\newcommand{\vweight}[2][]{\overline{\weight[#1]{#2}}}
\newcommand{\Tweight}[2][]{\Gamma^{#1(#2)}}
\newcommand{\lp}[2][]{p^{#1(#2)}}
\newcommand{\AvgM}{\mathbb{E} \mathcal{M}}
\newcommand{\tuple}[1]{\left(#1\right)}

\begin{document}
\title{Search problems in groups and branching processes}

\begin{abstract}
In this paper we study complexity of randomly generated instances of
Dehn search problems in finitely presented groups. We use Crump-Mode-Jagers
processes to show that most of the random instances are easy.
Our analysis shows that for any choice of a finitely presented  platform group in
Wagner-Wagner public key encryption protocol the majority of random keys
can be broken by a polynomial time algorithm.


\end{abstract}

\author{Pavel~Morar and Alexander~Ushakov}
\address{Mathematical Department\\
         Stevens Institute of Technology\\
         Hoboken, NJ 07030}
\email{pmorar,aushakov@stevens.edu}

\thanks{The work was partially supported by NSF grant DMS-0914773.}

\date{\today}
\maketitle


\section{Introduction}

Let $\CL$ be a set, typically an algebraic structure, and $\CP$ a property of objects in $\CL$.
Decision problems for $\CL$ are problems of the following nature: given
an object $\CO \in \CL$, decide whether $\CO$ has the property $\CP$, or not.
On the other hand, search problems are of the following nature: given
an object $\CO$ with the property $\CP$, find an {\em efficiently verifiable proof}
(sometimes called a ``witness") of the fact that $\CO$ has $\CP$. Typically,
the proofs must be verifiable in polynomial time by a deterministic Turing machine.
In this paper, motivated by applications in group-based cryptography~(\cite{Dehornoy_survey,MSU_book,MSU_book:2011}),
we study computational complexity of search problems of group theory:
word, conjugacy, and uniform membership problems for finitely presented
groups.

In classic complexity theory the time complexity $T(\CO)$
of an algorithmic problem  for a given input $\CO$ is the time required
by the algorithm (Solver) to find the answer for $\CO$.
The time complexity function $T$ measures the difficulty of the provided challenge $\CO$.
Note that the sets of positive instances of the word, conjugacy, and membership problems
are recursively enumerable and, hence, the problems are solvable,
i.e., there exists an algorithmic procedure which computes a required
proof for a given input in finite time.
However, in general, those sets are not recursive \cite{Novikov:1955,Boone:1958} and, hence,
this approach does not give a meaningful complexity estimate
on the running time because $T$ has no recursive upper-bound.

On the other hand, when we look at the problems from
the practical point of view, we assume that the instances of the
problem are somehow sampled by some procedure (Challenger), and the procedure
``knows'' that the sampled instance is a positive instance of the
problem, i.e., it has a proof that the instance is positive. This
spreads out the complexity  ``more evenly'' between two entities,
the one which generates a positive instance of the problem and the
one which finds a proof for that instance.
Thus, we treat a search problem here as a two-party game.
In this setting, a natural analysis of the problem is the comparison of running time of an
algorithm required for the challenger to generate an instance versus that
for the solver to find  a witness. We formally define this in Section \ref{se:mode_computations}.

\subsection{Group theory notation}

For a finite set $X$ denote by $X^{-1} = \{x^{-1} \mid x \in X \}$
the set of formal {\em inverses} of elements of $X$. The map $x
\rightarrow x^{-1} (x \in X)$ naturally extends to an involution
on the set $X^{\pm} = X \cup X^{-1}$ with
$(x^{-1})^{-1} = x$. By $(X^\pm)^\ast$ we denote the {\em free monoid}
on $X^{\pm}$ and by $F = F(X)$ the {\em free group} on $X$.
By $\varepsilon$ we denote the {\em empty word}, by '$=$' the equality relation and by $\FMnorm{w}$ the length of $w$ in the free monoid or the free group depending on the context.
For a word $w\in (X^\pm)^\ast$ by $\ovw$ and $\tdw$ we denote the free and cyclic reductions of $w$ correspondingly.
The word $w$ is \emph{reduced}
if $w = \Reduced{w}$ in $(X^\pm)^\ast$, and is \emph{cyclically reduced} if $w = \tdw$ in $(X^\pm)^\ast$.

For a subset $R \subseteq F(X)$ a pair $\gpr{X}{R}$ defines a group $F(X)/\ncl_F(R)$ denoted by $\group{X}{R}$ with the set of {\em generators} $X$ and the set of {\em relators} $R$. The pair $\gpr{X}{R}$ itself is called a \emph{group presentation} and is finite if both $X$ and $R$ are finite.
We say that a group $G$ has a presentation $\gpr{X}{R}$ if $G \simeq \group{X}{R}$, and $G$ is called \emph{finitely presented} if there is a finite group presentation for $G$.

For $w_1, w_2 \in F(X)$ we write $w_1 =_G w_2$ if they represent the same element of $G$ and $w_1 \conj{G} w_2$ if $w_1$ and $w_2$ are conjugate in $G$, that is, $w_1^g =_G w_2$ for some $g \in F(X)$, where $w_1^g = g^{-1} w_1 g$.

We say that $R \subseteq \FreeGroup{X}$ is \emph{symmetrized} if $R$ contains only cyclically reduced words
and is closed under taking inverses and cyclic permutations.
Denote by $R^{\star}$ the minimal symmetrized set containing $R$ (with all the words cyclically reduced). A presentation $\gpr{X}{R}$ is \emph{symmetrized} if $R = R^{\star}$. A finite presentation can be efficiently symmetrized and symmetrization does not change the computational properties of the fundamental problems (see~\cite{MSU_book:2011}).

\subsection{Dehn problems}
The following algorithmic questions are called the \emph{Dehn problems}. These problems are the central questions of combinatorial group theory and often are referred to as the fundamental problems for groups.

\medskip
\noindent\textbf{The word problem for $G$} is an algorithmic problem to decide if a given word $w \in \FreeGroup{X}$ represents the identity element of $G$.

\medskip
\noindent\textbf{The equivalence problem for $G$} is an algorithmic problem to decide if two given words $w_1, w_2 \in \FreeGroup{X}$ represent the same element of $G$, that is, if $w_1 =_G w_2$.

\medskip
\noindent\textbf{The conjugacy problem for $G$} is an algorithmic problem to
decide if given words $w_1, w_2  \in \FreeGroup{X}$ represent conjugate elements in $G$, that is, if $w_1 \conj{G} w_2$.

\medskip
\noindent\textbf{The (uniform) subgroup membership problem for $G$} is an algorithmic problem to
decide, given a tuple of words $h_1, \dots,h_k \in (X^\pm)^\ast$, if $h \in (X^\pm)^\ast$ represents an element of the subgroup $\gp{h_1,\ldots,h_k}$ in $G$.

\medskip
We say that a finitely presented group $G$ has a {\em decidable} word (equivalence, conjugacy, membership) problem if there exists an algorithm solving that problem.
The property of $G$ to have a decidable (or undecidable) word (equivalence, conjugacy, membership) problem is a group property, i.~e., it does not depend on a particular finite presentation of $G$. Note that the equivalence problem can be straightforwardly reduced to the word problem by changing the question from if $w_1 =_G w_2$ to if $w_1 w_2^{-1} =_G \varepsilon$. Search variations of Dehn problems are defined as the following.

\medskip
\noindent\textbf{The word search problem ($\WSP$) for $G$} is an algorithmic problem to find, given $w \in \FreeGroup{X}$ with $w =_G \varepsilon$, a witness of the fact that $w$ represents the identity in $G$.

\medskip
\noindent\textbf{The equivalence search problem ($\ESP$) for $G$} is an algorithmic problem to find, given $w_1, w_2 \in \FreeGroup{X}$ with $w_1 =_G w_2$, a witness of the fact that they are equivalent.

\medskip
\noindent\textbf{The conjugacy search problem ($\CSP$) for $G$} is an algorithmic problem to find, given conjugate in $G$ words $w_1, w_2  \in \FreeGroup{X}$, a witness of $w_1 \conj{G} w_2$.

\medskip
\noindent\textbf{The (uniform) subgroup membership problem ($\MSP$) for $G$} is an algorithmic problem to find, given a tuple of words $h_1, \dots,h_k \in \FreeGroup{X}$ and $h \in \FreeGroup{X}$ with $h \in \gp{h_1,\ldots,h_k}$ in $G$, a witness of $h \in \gp{h_1,\ldots,h_k}$ in $G$.

\medskip
We use the same witnesses as in~\cite{MU1}. For examples of witnesses see~\cite{Shpilrain:2010}.

\subsection{The Wagner-Magyarik cryptosystem and its modifications}
In the 1950's it was proven that finitely presented groups can have undecidable word problems, see \cite{Novikov:1955,Boone:1958} (see also~\cite{Britton:1963,Borisov:1969,Collins:1986,Myasnikov-Osin:2011}).

The hardness of the fundamental problems of combinatorial group theory
inspired many cryptographic constructions. We are particularly interested in the hardness of search variations of Dehn problems. One of the cryptoschemes inspired by the fundamental problems of combinatorial group theory, and the most interesting to us, is the Wagner-Magyarik public-key cryptosystem proposed in \cite{MW}, where the authors outline a conceptual construction of a cryptosystem based on the word problem, and illustrate their proposal with a specific suggestion for the choice of the system parameters. Here is an outline of the construction.
\begin{algorithm}{Wagner-Magyarik PKC}
\label{al:WM_protocol}
\begin{algorithmic}[1]
\InitialSetup
Choose a finitely presented group $G = \group{X}{R}$ with a computationally hard word problem and its quotient $G' = \group{X}{R \cup S}$ with a computationally easy word problem.
Choose words $w_0, w_1$ representing different elements in $G'$.
\PublicKey The triple $(\gpr{X}{R},w_0,w_1)$.
\PrivateKey The set $S$.
\Encoding
To encrypt $b\in\{0,1\}$ randomly rewrite a word $w_b$
(see Algorithm \ref{A:WMRandomEqualWord} below)
to obtain a random word $w$ satisfying $w=_Gw_b$.
\Decoding
Since $G'$ is a quotient of $G$ and $w_0 \ne_{G'} w_1$ it follows that
$w =_G w_b$ if and only if $w =_{G'}w_b$.
Hence, to decrypt $w$ it is sufficient to check if $w=_{G'}w_0$ or $w=_{G'}w_1$.
\end{algorithmic}
\end{algorithm}

The scheme received some critique especially for being vague and missing a lot of important details (see \cite{Birget-Magliveras-Sramka:2006}). Also it was shown to be vulnerable to reaction attacks (see \cite{Vasco-Steinwandt:2004}).
In addition, it was observed in \cite{Birget-Magliveras-Sramka:2006} that
security of this scheme depends on the hardness of the \emph{word choice problem}.

\medskip
\noindent\textbf{The word choice problem for $G$:}
Given words $w_0,w_1,w$ decide if $w_0=_G w$ or $w_1 =_G w$,
provided that exactly one equality holds.

\medskip
The word choice problem and word problem for $G$ are not equivalent.
In particular, the word problem can be undecidable while the word choice problem is always decidable. It can be attacked by solving the word search problem
for $ww_0^{-1}$ and, in parallel, for $ww_1^{-1}$.
Exactly one of those words is trivial and
only one process stops giving a witness for the corresponding choice.

A very important part of this scheme is the generation of a random word $w$.
In \cite{MW} the following algorithm was outlined.

\begin{algorithm}{Random Equal Word($\gpr{X}{R}, w, n$)}
\label{A:WMRandomEqualWord}
\begin{algorithmic}[1]
	\Require A finite presentation $\gpr{X}{R}$, a word $w \in \FreeGroup{X}$, and $n \in \MN$.
	\Ensure A word $w' \in \FreeGroup{X}$ such that $w' =_G w$.
	\State $\ElementaryIdentities = \Set{xx^{-1}, \, x^{-1}x}{x \in X} \cup R$.
	\State $w_0 = w$.
	\For{$i = 1$ to $n$}
		\State Randomly perform one of the following: either insert a random $u \in \ElementaryIdentities$ into $w_{i-1}$ at a random
position $p$, or remove some random occurrence of some $u \in \ElementaryIdentities$ in $w_{i-1}$.
		\State Call the obtained word $w_i$.
	\EndFor
	\State \Return $\Reduced{w_n}$.
\end{algorithmic}
\end{algorithm}

Note that step $4$ of this algorithm is not completely specified.
It does not say how to make required random choices.
Mathematical foundations of the protocol were never analyzed mostly because of the vagueness of the scheme.
In this paper we do a very general mathematical analysis assuming that random positions are chosen uniformly.
Our analysis does not depend on a choice of the public and private information, namely
on the choice of $X,R,S,w_0,w_1$.
Instead we investigate characteristics of the words $w_0$, \dots, $w_n$ generated in the protocol.

Even though the Wagner-Magyarik scheme is considered to be insecure
it is still being discussed and different variations are being proposed.
For instance, in \cite{Birget-Magliveras-Sramka:2006}
the authors consider ways to make the Wagner-Magyarik scheme viable
by (considerably) changing the design and the platform group.
In \cite{Abisha-Thomas-Subramanian:2003,Vehel-Perret:2006,Vehel-Perret:2010}
the authors study Wagner-Magyarik-like schemes based on the
word choice problem in semigroups.

The original Wagner-Magyarik cryptosystem can be modified in many ways.
In particular, one can employ the hardness
of the conjugacy problem as described below.

\begin{algorithm}{WM-PKC based on the conjugacy problem}
\label{al:CP_protocol}
\begin{algorithmic}[1]
\InitialSetup
Choose a finitely presented group $G=\group{X}{R}$ with a computationally hard conjugacy problem
and its quotient $G' = \group{X}{R \cup S}$ with an easy conjugacy problem.
Choose words $w_0,w_1$ representing non-conjugate elements of $G'$.
\PublicKey The triple $(G,w_0,w_1)$.
\PrivateKey The set $S$.
\Encoding
To encrypt $b\in\{0,1\}$ randomly rewrite a word $w_b$
and obtain a word $w$ satisfying $w\conj{G} w_b$.
\Decoding
Since $G'$ is a quotient of $G$ and $w_0\not\conj{G'} w_1$ it follows that
$w\conj{G} w_b$ if and only if $w\conj{G'}w_b$.
Hence to decrypt $w$ it is sufficient to check if
$w\conj{G'}w_0$ or $w\conj{G'}w_1$.
\end{algorithmic}
\end{algorithm}

We can use the membership search problem in a similar way.
Note that the reaction attack of Vasco and Steinwandt applies  to the
both variations of the original protocol.
Nevertheless these modifications are interesting in their own right.

\subsection{Mode of computations and main results}
\label{se:mode_computations}

All computations are assumed to be performed on a random access machine.
We use notation $\tO(n^c)$ to denote the class of functions $\bigcup_{k \ge 0} O(n^c \ln^k(n))$.

Now we formalize the challenger-solver game analysis. Let $D \subseteq \FreeGroup{X}$ be the set of positive instances of some problem and $\{\mu_{n}\}_{n \ge 0}$ be a system of probability measures (distributions) on $D$. We assume that these measures are given in a way that it is easy to sample elements according to them (for example, by an efficient algorithm). The index $n$ is considered to be a complexity parameter, so $\mu_{n}$ gives a probability distribution on the subset of $D$ of instances of complexity $n$. For a given $n$ the challenger generates a random instance $d$ according to $\mu_{n}$ and sends it to a solver $\CA$. Let $T_\CA(d)$ be the time spent by $\CA$ on $d$. We say that the solver $\CA$ solves the randomized search problem $(D,\{\mu_{n}\}_{n \ge 0})$
\emph{generically} in time $T(n)$ if:
\[
	\mu_{n} \rb {d \in D \mid T_\CA(d)\le T(n)} \underset{n \to \infty}{\rightarrow} 1.	
\]
For more on generic case complexity see~\cite{KMSS1,GMMU:2007}.

In this paper we consider particular generators inspired by Algorithm~\ref{A:WMRandomEqualWord}, namely, Algorithm~\ref{A:esp_generator} for $\WSP$ and $\ESP$,
Algorithm~\ref{A:csp_generator} for $\CSP$,
Algorithm~\ref{A:msp_generator} for $\MSP$. The solvers $\CA_{WP}$, $\CA_{CP}$ (\cite{Ushakov:thesis, MSU_book:2011}), and $\CA_{MP}$ are discussed in Section~\ref{se:SearchAlgorithms}.
The main results of this paper are the following theorems, which are proven in Section~\ref{Sec:diagramsAndTrees}.

\begin{thma}{\textbf{A.}}
\emph{For any finite presentation $\gpr{X}{R}$ Algorithm $\CA_{WP}$
solves the randomized problem $(\WSP, \{\mu_n\}_{n \ge 0})$ defined by Algorithm~\ref{A:esp_generator} generically in polynomial time $\tilde{O}\rb{n^{1 + e^2 \ln L(R)}}$.
}
\end{thma}

\begin{thma}{\textbf{B.}}
\emph{For any finite presentation $\gpr{X}{R}$ and $w\in \FreeGroup{X}$
Algorithm $\CA_{WP}$ solves the randomized problem
$\left(\ESP(w), \{\mu_{n, w}\}_{n \ge 0}\right)$ defined by Algorithm~\ref{A:esp_generator} generically in polynomial time $\tilde{O}\rb{(\FMnorm{w} + n) n^{e^2 \ln L(R)}}$.
}
\end{thma}

\begin{thma}{\textbf{C.}}
\emph{
For any finite presentation $\gpr{X}{R}$ and $w\in \FreeGroup{X}$ Algorithm $\CA_{CP}$ solves the randomized problem
$\left(\CSP(w), \{\nu_{n, w}\}_{n \ge 0} \right)$ defined by Algorithm~\ref{A:csp_generator} generically in polynomial time $\tilde{O}\rb{\FMnorm{\CReduced{w}}(\FMnorm{\CReduced{w}} + n) n^{2 e^2 \ln L(R)}}$.
}
\end{thma}

\begin{thma}{\textbf{D.}}
\emph{
For any finite presentation $\gpr{X}{R}$ and a finite set $H \subset \FreeGroup{X}$ Algorithm $\CA_{MP}$ solves the randomized problem $\left(\MSP(H), \{\rho_{k, n, H}\}_{k, n \ge 0}\right)$ defined by Algorithm~\ref{A:msp_generator0} generically in polynomial time $\tilde{O}\rb{(k + n)n^{e^2 \ln L(R)}}$.
}
\end{thma}

\begin{thma}{\textbf{E.}}
\emph{
For any finite presentation $\gpr{X}{R}$ and a finite set $H \subset \FreeGroup{X}$ Algorithm $\CA_{MP}$ solves the randomized problem $\left(\MSP(H), \{\rho'_{n, q, H}\}_{n \ge 0}\right)$ defined by Algorithm~\ref{A:msp_generator} generically in polynomial time $\tilde{O}\rb{n^{1 + e^2 \ln L(R)}}$.
}
\end{thma}

Similar analysis was done in~\cite{MU1,Ushakov:thesis} for a different type of challengers. The algorithms in~\cite{MU1,Ushakov:thesis} generate words in $(X^\pm)^\ast$, i.e., nonreduced words. Here all generated words are reduced, therefore, this work is a significant improvement over~\cite{MU1,Ushakov:thesis}.

\subsection{Outline}

In Section~\ref{S:generators} we describe particular challengers for the word, equivalence, conjugacy and membership search problems. The solvers are discussed in Section~\ref{se:SearchAlgorithms}. In Section~\ref{Sec:diagramsAndTrees} we consider random trees associated with random instances generated by the challengers and use them to prove the main results. In the proofs we use Crump-Mode-Jagers processes (see the appendix for the overview).


\section{Random instances of search problems}
\label{S:generators}

In this section we formalize the key generation procedure (Algorithm~\ref{A:WMRandomEqualWord}) by making each word transformation explicit. In Sections~\ref{Sub:csp_generator} and \ref{Sub:msp_generator} we propose similar procedures for generation of random conjugates and random elements of finitely generated subgroups.

We use the following notation throughout the paper.
For $n \in \MN$ by $\URandom{n}$ denote a uniformly random element of the set $\{0, \dots, n\}$.
For any distribution $\DRandom{S}$ on a set $S$ we denote by $u \gets \DRandom{S}$ an element $u$ sampled according to $\DRandom{S}$. For $S \subseteq (X^\pm)^\ast$ or $S \subseteq \FreeGroup{X}$ we denote by $\LDRandom{S}$ the distribution of the length $\FMnorm{w}$ induced by $\DRandom{S}$.

For a fixed a finite presentation $\gpr{X}{R}$ of a group $G$ define the set of {\em elementary identities}:
\begin{equation}\label{E:TGSetDef}
\ElementaryIdentities = \Set{xx^{-1}, \, x^{-1}x}{x \in X} \cup R \subset (X^\pm)^\ast,
\end{equation}
and fix an arbitrary distribution $\DRandom{\ElementaryIdentities}$ on $\ElementaryIdentities$. A \emph{random $\ElementaryIdentities$-transformation} of a word $w$ is an insertion of a word $u \gets \DRandom{\ElementaryIdentities}$ into $w$ at the position $p \gets \URandom{\FMnorm{w}}$ (without cancelation).

\subsection{Equivalent words.}
\label{Sub:wsp_generator}
The following algorithm generates a random word equivalent to $w$ in $G$.
\begin{algorithm}{RandomEqualWord($\gpr{X}{R}, w, n$)}
\label{A:esp_generator}
\begin{algorithmic}[1]
	\Require A finite presentation $\gpr{X}{R}$, a word $w \in \FreeGroup{X}$, and $n \in \MN$.
	\Ensure A word $w'$ equivalent to $w$ in $G$.
	\State $w_0 = w$.
	\For{$i = 1$ to $n$}
		\State Apply a random $\ElementaryIdentities$-transformation to $w_{i-1}$ to get $w_i$.
	\EndFor
	\State \Return $\Reduced{w_n}$.
\end{algorithmic}
\end{algorithm}

Note that unlike Algorithm~\ref{A:WMRandomEqualWord} Algorithm~\ref{A:esp_generator} does not explicitly remove relators from $w_{i-1}$ (if they occur in $w_{i-1}$). It does that implicitly by inserting an inverse $r^{-1}$ next to an occurrence of $r \in \ElementaryIdentities$ in $w_{i-1}$. If the subword $r \circ r^{-1}$ (or $r^{-1} \circ r$) is not changed and is present in $w_n$, then the free cancelation on step $5$ removes it.

Now, for a word $w \in \FreeGroup{X}$ define a set:
\[
	\ESP(w) = \Set{w' \in \FreeGroup{X}}{w' =_G w}.
\]
Clearly, Algorithm~\ref{A:esp_generator} generates elements of $\ESP(w)$ on the input $w$ and for every $n \in \MN$ it defines a probability measure $\mu_{n,w}$ on $\ESP(w)$:
\[
    \mu_{n,w}(w') = \Prob{\text{Algorithm~\ref{A:esp_generator} generates $w'$ from $w$ in $n$ steps}}.
\]
The \emph{support} of $\mu_{n,w}$ is the finite set:
\[
	\supp(\mu_{n,w}) = \Set{w' \in \FreeGroup{X}}{\mu_{n,w}(w') > 0} \subseteq \ESP(w).
\]
Let us point out some properties. In general, $\mu_{n,w}$ is not uniform on $\supp(\mu_{n,w})$ and for a fixed $w$ the sets $\{\supp(\mu_{n,w})\}_{n \ge 0}$ are not disjoint because a word can be generated in several different ways (cf.~\cite{Elder_Rechnitzer_Rensburg:2013}).
If $\supp{\DRandom{\ElementaryIdentities}} = \ElementaryIdentities$, then:
\begin{equation}
\label{Eq:muSupport}
	\ESP(w) = \bigcup_{n \in \MN} \supp(\mu_{n,w}).
\end{equation}
For $w = \varepsilon$ we use the notation $\mu_n$ for $\mu_{n, \varepsilon}$ and $\WSP$ for $\ESP(\varepsilon)$.

\subsection{Random conjugates}
\label{Sub:csp_generator}

In this section we define a generator of random conjugates of a given word $w$ in $G$.
The most straightforward way to generate a conjugate of $w$ is to conjugate
$w$ in the free group $F(X)$ to obtain $c^{-1}wc$ and then apply Algorithm \ref{A:esp_generator} to $c^{-1}wc$.
Under natural assumptions on the choice of $c$ we will be able to generate every conjugate of $w$.
Nevertheless we prefer another approach because this one has the following bias.
If $u$ and $v$ are cyclic permutations of each other, then the distributions defined
for $u$ and $v$ are not the same (an annular diagram constructed for, say $u$,
has a long tail attached to the beginning/end of the cyclic $u$).
That is the reason why we consider another  generation method.

\begin{algorithm}{RandomConjugate($\gpr{X}{R}, w, n$)}
\label{A:csp_generator}
\begin{algorithmic}[1]
	\Require A finite presentation $\gpr{X}{R}$, a word $w \in \FreeGroup{X}$ with $w \neq \varepsilon$, and $n \in \MN$.
	\Ensure A word $w'$ conjugate to $w$ in $G$.
    \State Cyclically reduce $w$.
	\State Split $\CReduced{w} = w_0 \circ x$ with $x \in X^\pm$.
	\State $u \gets \REW(\gpr{X}{R}, w_0, n)$.
	\State \Return A random uniformly chosen cyclic permutation of $\CReduced{u \circ x}$.
\end{algorithmic}
\end{algorithm}
The first and last positions of $\CReduced{w}$ correspond to the same position of $\CReduced{w}$ as a cyclic word. To avoid counting it twice we perform step~$2$. Another way to think of $\CReduced{w}$ is as of an annular diagram boundary word. Its first and last positions correspond to the same point on the boundary.
So to pick a uniformly random position on the boundary is the same as to pick a uniformly random position of the word $w_0$.

For a word $w$ Algorithm \ref{A:csp_generator} generates elements of the set:
\[
	\CSP(w) = \Set{w' \in \FreeGroup{X}}{w' \conj{G} w \text{ and } w' = \CReduced{w'}}
\]
and for $n \in \MN$ it defines a probability measure $\nu_{n, w}$ on $\CSP(w)$.
It is easy to see that $\nu_{n, u}=\nu_{n, v}$ for words $u, v$ conjugate in the corresponding free group. If $\supp{\DRandom{\ElementaryIdentities}} = \ElementaryIdentities$, by~\eqref{Eq:muSupport} it holds:
\[
	\CSP(w) = \bigcup_{n \in \MN} \supp(\nu_{n, w}).
\]

\subsection{Random subgroup elements}
\label{Sub:msp_generator}
Let $H = \{h_1, \ldots, h_k\}$, $H^{\pm} = H \cup H^{-1}$, and $\gp{H}$ be a subgroup of $G$.
Fix an arbitrary distribution $\DRandom{H^{\pm}}$ on $H^{\pm}$ so we are able to sample random elements from $H^{\pm}$.

A word $w$ represents an element of $\gp{H}$ if and only if it is equal in $G$ to a product $w_h$ of elements from $H^{\pm}$.
The most straightforward way to generate such a word is to pick a product $w_h$ and apply a sequence of $\ElementaryIdentities$-transformations to it.
 We formalize this approach in the following algorithm.
\begin{algorithm}{RandomSubgroupWord($\gpr{X}{R}, H, k, n$)}
\label{A:msp_generator0}
\begin{algorithmic}[1]
	\Require A finite group presentation $\gpr{X}{R}$, a finite set $H\subseteq F(X)$, $k, n \in \MN$.
	\Ensure A word $w' \in \gp{H}$ in $G$.
	\State $v \gets v_1 \circ v_2 \circ \dots \circ v_k$ with $v_i \gets \DRandom{H^{\pm}}$.
	\State $u \gets \REW(\gpr{X}{R}, v, n)$.
	\State \Return $\Reduced{u}$
\end{algorithmic}
\end{algorithm}
Define the set:
\[
	\MSP(H) = \Set{w' \in \FreeGroup{X}}{w' \in \gp{H} \text{ in } G}.
\]
Clearly Algorithm \ref{A:msp_generator0} generates elements of $\MSP(H)$ and for $k, n \in \MN$
it defines a probability measure $\{\rho_{k,n,H}\}_{k, n \ge 0}$ on $\MSP(H)$.
If $\supp{\DRandom{\ElementaryIdentities}} = \ElementaryIdentities$ and $\supp{\DRandom{H^{\pm}}} = H^{\pm}$, then:
\[
	 \MSP(H) = \bigcup_{k, n\in\MN}\supp(\rho_{k,n,H}).
\]

We can also use another approach, which iteratively builds up a word by expanding its base in $H$ (attaching $h \in H^{\pm}$ to the end of the word) and increasing its complexity in $G$ (applying $\ElementaryIdentities$-transformations). The parameter $q$ of the following algorithm defines which type of operations we favor more.
\begin{algorithm}{RandomSubgroupWord2($\gpr{X}{R}, H, n, q$)}
\label{A:msp_generator}
\begin{algorithmic}[1]
	\Require A finite group presentation $\gpr{X}{R}$, a finite set $H\subseteq F(X)$, $n \in \MN$, $q \in (0,1)$.
	\Ensure A word $w' \in \gp{H}$ in $G$.
	\State $w_0 = \varepsilon$
	\For{$i = 1$ to $n$}
		\State With probability $q$ set $w_{i} \gets w_{i-1} \circ u$ with $u \gets \DRandom{H^{\pm}}$ or otherwise apply a random $\ElementaryIdentities$-transformation to $w_{i-1}$ to get $w_i$.
	\EndFor	
	\State \Return $\Reduced{w_n}$
\end{algorithmic}
\end{algorithm}
It is easy to see that if $H = \emptyset$ or $q = 0$, then Algorithm~\ref{A:msp_generator} is equivalent to Algorithm~\ref{A:esp_generator}.
Clearly, Algorithm \ref{A:msp_generator} generates elements of $\MSP(H)$ and for $n \in \MN$, $q \in (0,1)$
it defines a probability measure $\rho'_{n, q, H}$ on $\MSP(H)$.
If $\supp{\DRandom{\ElementaryIdentities}} = \ElementaryIdentities$ and $\supp{\DRandom{H^{\pm}}} = H^{\pm}$, then:

\[
	 \MSP(H) = \bigcup_{n\in\MN}\supp(\rho'_{n, q, H}).
\]


\section{Search algorithms for finitely presented groups}
\label{se:SearchAlgorithms}

There are several general techniques for solving search problems in groups.
All the search problems under consideration are recursively enumerable and hence
can be solved by a total enumeration using relators of a group presentation.
Also, one can use a version of coset enumeration (the Todd-Coxeter algorithm, see~\cite{ToddCoxeter:1936})
or the Knuth-Bendix algorithm (\cite{Gilman:1979,Epstein:1991}).
Both algorithms can be used to solve $\WSP$, but were originally designed for other purposes.
The Todd-Coxeter algorithm attempts to construct the Cayley graph (or, more generally, the Schreier graph) of a group $G$.
The Knuth-Bendix algorithm attempts to find a complete rewriting system for a given group presentation.
There are no known (to the authors) complexity upper bounds for these algorithms
in the context of all finitely presented groups.

In this paper we use algorithms proposed in~\cite{Ushakov:thesis} that were
specifically designed to solve $\WSP$ and $\CSP$ in finitely presented groups.
Here we use slightly different notation and denote Algorithm $\CA$ of \cite{MSU_book:2011} solving $\WSP$ by $\CA_{WP}$
and Algorithm $\CC$ solving $\CSP$ by $\CA_{CP}$.
In Section~\ref{Sub:msp_solver} we introduce Algorithm $\CA_{MP}$ (similar to $\CA_{WP}$) to solve $\MSP$.
The time complexity of these algorithms depends on the notion of {\em depth},
which measures complexity of input words and is defined as a parameter of the corresponding diagrams.
In the next section we shortly review basic definitions for diagrams and depth
(see~\cite{Lyndon-Schupp:2001,Ol_book,Rileybook2007}) and
discuss the time complexity of Algorithms~$\CA_{WP}$ and $\CA_{CP}$.

\subsection{Diagrams}
\label{Sub:diagrams_preliminaries}

For a set $S \subset \EP$ let $\partial S$ be its boundary
and $\overline{S}$ the closure of $S$ in $\EP$. Let $D$ be a finite connected planar $X$-digraph with set of vertices $V(D)$ and set of edges $E(D)$. Let $C(D)$ be a set of \emph{cells} of $D$ which are connected and simply connected bounded components of $\MR^2\setminus D$. The unbounded component of $\MR^2\setminus D$ is called the \emph{outer cell} of $D$ denoted by $c_{\text{out}}$. An edge $e \in E(D)$ is \emph{free} if it does not belong to $\partial c$ for any $c \in C(D)$. For any $e \in E(D)$ we denote its label by $\mu(e) \in X^{\pm}$. The boundary of a cell $c \in C(D)$ traversed in a counterclockwise direction starting from some vertex of $c$ makes a closed path $e_1 \dots e_n$ giving the word $\mu(c) = \mu(e_1) \dots \mu(e_n) \in (X^{\pm})^\ast$ called a \emph{boundary label} of $c$. Depending on a starting vertex we get a cyclic permutation of the same word.

For the rest of this subsection let $D$ be a finite connected planar $X$-digraph with a \emph{base vertex} $v_0 \in V(D) \cap \partial c_{\text{out}}$.
The graph $D$ is a \emph{van Kampen diagram} over $\gpr{X}{R}$ if $\mu(c) \in R^{\star}$ for every $c \in C(D)$.
The boundary label $\mu(D)$ of $D$ is the boundary label of $\partial c_{\text{out}}$ read starting from $v_0$ in a counterclockwise direction.
Note that we need also to specify the first edge to read from $v_0$, that is, the starting boundary position, but it is not important for our considerations so we omit this issue.
\begin{lemma}[van Kampen lemma]
A word $w \in (X^{\pm})^\ast$ represents the identity of the group $\group{X}{R}$ if and only if
there exists a van Kampen diagram over $\gpr{X}{R}$ with $\mu(D) = w$.
\end{lemma}

We generalize van Kampen diagrams to the case of subgroup elements.
Let $H$ be a finite generating set of a subgroup $\gp{H}$ of $\group{X}{R}$. The graph $D$ is a \emph{generalized van Kampen diagram} over $\gp{H} \le \group{X}{R}$ if for every $c \in C(D)$ one of the following holds: either $\mu(c) \in R^{\star}$ or $v_0 \in \partial c$ and the label $\mu(c)$ read starting from $v_0$ belongs to $H^\pm$.
We call the first type of cells \emph{$R$-cells} and the second type \emph{$H$-cells}. The boundary label of $D$ is defined in the same way as for van Kampen diagrams. It is easy to see that $w$ represents an element of $\gp{H}$ in $G$ if and only if there exists a generalized van Kampen diagram $D$ over $\gp{H} \le \group{X}{R}$ with $\mu(D) = w$.

Now let us exclude one of the cells from $C(D)$ and call it the \emph{inner cell} $c_{\text{in}}$ of $D$. Denote $v_0$ by $v_{\text{out}}$ and pick any vertex $v_{\text{in}} \in V(D) \cap \partial c_{\text{in}}$. We call $D$ an \emph{annular (Schupp) diagram} (see~\cite{Schupp:1968}) over $\gpr{X}{R}$ if $\mu(c) \in R^{\star}$ for any $c \in C(D)$. Its two boundary labels $\mu_{\text{in}}(D) = \mu(c_{\text{in}})$ and $\mu_{\text{out}}(D) = \mu(c_{\text{out}})$ read in a counterclockwise direction from $v_{\text{in}}$ and $v_{\text{out}}$ correspondingly, are called the \emph{inner} and \emph{outer} labels of $D$. For any $w_1, w_2 \in (X^\pm)^\ast$ we have that $w_1 \conj{G} w_2$ if and only if there exists an annular diagram $D$ over $\gpr{X}{R}$ with $\mu_{\text{in}}(D) = w_1$ and $\mu_{\text{out}}(D) = w_2$.

We measure diagram complexity using a notion of depth (introduced in~\cite{MSU_book:2011}). For a (van Kampen, generalized van Kampen, or annular) diagram $D$ define the {\em dual graph} $D^\ast = (V^\ast,E^\ast)$ as an undirected graph with $V^\ast = C(D) \cup c_{\text{out}}$ (for annular diagrams we add $c_{\text{in}}$) and $E^\ast = \{(c_1,c_2) \mid \partial c_1 \cap \partial c_2 \ne\emptyset \}$. We denote the graph distance in $D^\ast$ by $d^\ast$.

The \emph{depth} of a (generalized) van Kampen diagram $D$ is defined by:
\[
    \delta(D) = \max_{c \in C(D)} d^\ast(c, c_{\text{out}}).
\]
The \emph{depth} of an annular diagram $D$ is:
\[
	\delta(D) = \max_{c \in C(D)} \left[ \min \left( d^\ast(c, c_{\text{out}}), d^\ast(c, c_{\text{in}}) \right) \right].	
\]
\begin{remark}
There is a similar notion of a diagram radii (see~\cite{Gersten-Riley:2002, Rileybook2007}).
\end{remark}

Define the \emph{depth} of a word $w \in \FreeGroup{X}$ as:
\[
	\delta(w) = \min_{\substack{D \text{ is}\\
								\text{a van Kampen}\\
								\text{diagram}}}
					\Set{\delta(D)}{\mu(D) = w}
\]
if $w =_G \varepsilon$ and $\delta(w) = \infty$ otherwise, the \emph{conjugate depth} of two words $w_1, w_2 \in \FreeGroup{X}$ as:
\[
	\delta_{\conj{}}(w_1, w_2) =
		\min_{\substack{D \text{ is}\\
						\text{an annular}\\
						\text{diagram}}}
		\Set{\delta(D)}{\mu_{\text{in}}(D) = w_1, \mu_{\text{out}}(D) = w_2}
\]
if $w_1 \conj{G} w_2$ and $\infty$ otherwise, and the \emph{depth} of a word $w \in \FreeGroup{X}$ with respect to a finite set $H \subset (X^\pm)^\ast$ as:
\[
\delta_H(w) =
		\min_{\substack{D \text{ is}\\
						\text{a generalized}\\
						\text{van Kampen}\\	
						\text{diagram}}}
		\Set{\delta(D)}{w(D) = w}
\]
if $w \in \gp{H}$ in $G$ and $\infty$ otherwise.

Recall that by $\tilde{O}$ we denote the soft-mod complexity introduced in Section \ref{se:mode_computations}. Set $L(R) = \sum_{r\in R} |r|$.

\begin{theorem}[Theorem~16.4.3 in~\cite{MSU_book:2011}]
\label{Thm:wspSolverComplexity}
Let $G$ be a group given by a finite symmetrized presentation $\gpr{X}{R}$ and $w \in \FreeGroup{X}$.
Algorithm $\CA_{WP}$ stops on the input $\gpr{X}{R}$, $w$ if and only if $w =_G \varepsilon$.
Furthermore, it terminates in at most $\delta(w)$ iterations and the time complexity of Algorithm $\CA_{WP}$ is bounded above by:
\[
	\tilde{O}\rb{ \Fnorm{w} L(R)^{\delta(w)}}.
\]
\end{theorem}

\begin{theorem}[Theorem~17.6.12 in~\cite{MSU_book:2011}]
\label{Thm:cspSolverComplexity}
Let $G$ be a group given by a finite symmetrized presentation $\gpr{X}{R}$ and $w_1, w_2 \in \FreeGroup{X}$.
Algorithm $\CA_{CP}$ stops on the input $\gpr{X}{R}$, $w_1$, $w_2$ if and only if $w_1 \conj{G} w_2$.
Furthermore, it terminates in at most $\delta_{\conj{}}(w_1, w_2)$ iterations and the time complexity of Algorithm $\CA_{CP}$ is bounded above by:
\[
	\tilde{O}\rb{ \Fnorm{w_1} \Fnorm{w_2} L(R)^{2 \delta_{\conj{}}(w_1, w_2)}}.
\]
\end{theorem}

\subsection{Algorithm for the uniform membership search problem.}
\label{Sub:msp_solver}
To solve the membership search problem we use finite inverse $X$-digraphs, an operation called {\em $R$-completion}, and Stallings' folding. We assume that the presentation $\gpr{X}{R}$ is symmetrized.

Any finite $X$-digraph $\Gamma$ with a fixed {\em base-vertex} $v_0$ can be viewed as
a finite state automaton {\em accepting} the language:
\[
	\CL(\Gamma) = \Set{\mu(p)}{p \text{ is a loop in $\Gamma$ at }v_0}.	
\]
From a given $X$-digraph $\Gamma$ one can construct a new automaton $\CC(\Gamma)$ by adding for every $r \in R$ a loop labeled by $r$ at every state $u \in \Gamma$.
By an \emph{$R$-completion} of $\Gamma$ we understand a computation of
$\CC^{k}(\Gamma)$ for some $k\in\MN$.
The following properties of $\CC(\Gamma)$ follow immediately from the construction.

\begin{proposition}
\label{pr:completion_properties}
For every $\gpr{X}{R}$ and $\Gamma$ the following holds:
\begin{itemize}
    \item[(a)]
$\Gamma$ is a subgraph of $\CC(\Gamma)$.
    \item[(b)]
$\gamma(\CL(\Gamma)) = \gamma(\CL(\CC(\Gamma)))$, where $\gamma:F(X) \rightarrow \group{X}{R}$ is a canonical epimorphism.
    \item[(c)]
$|\CC(\Gamma)| \le |\Gamma| \cdot L(R)$.
\qed
\end{itemize}
\end{proposition}

For a word $w = w_1 \ldots w_n$ define the $X$-digraph $\Gamma(w)$
as a sequence of edges labeled with the letters of $w$ as shown
in Figure~\ref{fi:Gw_graph}.
The first vertex of $\Gamma(w)$ is denoted by $v_0$ and the last one by $v_k$. The vertex $v_0$ is the base vertex of $\Gamma(w)$.
\begin{figure}[h]
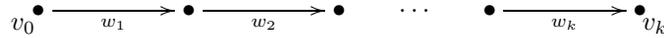

\centerline{
\xygraph{
!{<0cm,0cm>;<2cm,0cm>:<0cm,2cm>::}
!{(0,0)}*+{\bullet}="0"
!{(1,0)}*+{\bullet}="1"
!{(2,0)}*+{\bullet}="2"
!{(3,0)}*+{\bullet}="3"
!{(4,0)}*+{\bullet}="4"
"0":@[|(1.5)]"1"_{w_1}
"1":@[|(1.5)]"2"_{w_2}
"3":@[|(1.5)]"4"_{w_k}
!{(2.5,0)}*+{\ldots}="6"
!{(-.1,-.1)}*+{v_0}
!{(4.1,-.1)}*+{v_k}
}}
\caption{\label{fi:Gw_graph}The graph $\Gamma(w)$.}
\end{figure}

For $h_1, \ldots, h_k, w\in F(X)$ define a graph $\Gamma(w, h_1, \ldots, h_k)$ to be a wedge graph of $n$ loops labeled with words $h_1, \ldots, h_k$
and the graph $\Gamma(w)$ shown in Figure \ref{fi:Gws_graph}.
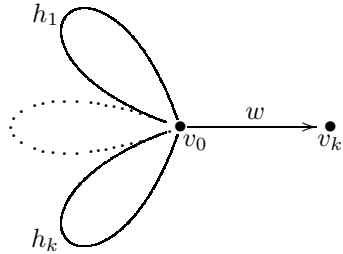
\begin{figure}[ht]
\centerline{
\begin{xy}
(0,0)*+{\bullet}="A";
(0,0)*+{}="A1";
**\crv{(-10,30)&(-30,10)}; \POS?(.5)*_+!R{h_{1}},"A1";
**\crv{(-10,-30)&(-30,-10)}; \POS?(.5)*_+!R{h_{k}},"A1";
**\crv{~*=<4pt>{.}(-30,-12)&(-30,12)}; \POS?(.5)*_+!R{},"A1";
(0,0)*+{}="A";
(20,0)*+{\bullet}="B";
{\ar@{->}"A";"B"};\POS?(.5)*_+!D{w},"B";
(2,-2)*+{v_0};
(20,-2)*+{v_k};
\end{xy}
}
\caption{\label{fi:Gws_graph}The graph $\Gamma(w, h_1, \ldots, h_k)$.}
\end{figure}

The next algorithm solves the uniform membership search problem for finitely generated subgroups of finitely presented groups.
By $S$ we denote the Stallings' folding of an $X$-digraph (see~\cite{Kapovich_Miasnikov:2002}).

\begin{algorithm}{(Uniform) Membership search problem solver $\CA_{MP}$.}
\label{A:DecisionMSP}
\begin{algorithmic}[1]
\Require A finite symmetrized presentation  $\gpr{X}{R}$ and words $w,h_1, \dots, h_k \in \FreeGroup{X}$.
\Ensure $YES$ if $w \in \gp{h_1, \ldots, h_k}$ in $G$ and a finite $X$-digraph $\Gamma_{n}$ which accepts $w$.
\State $\Gamma_0 = \Gamma(w, h_1, \ldots, h_k)$.
\While{$v_0 \ne v_k$ in $\Gamma_i$}
    \State $\Gamma_i \gets S(\CC(\Gamma_{i-1}))$.
\EndWhile
\State \Return $YES$ and the obtained graph $\Gamma_n$.
\end{algorithmic}
\end{algorithm}

The graph $\Gamma_n$ is a witness for the fact that $w \in \gp{h_1, \dots, h_k}$ in $G$.

\begin{theorem}
\label{Thm:mspSolverComplexity}
Let $G$ be a group given by a finite symmetrized presentation $\gpr{X}{R}$ and $w, h_1, \ldots, h_k \in \FreeGroup{X}$.
Algorithm \ref{A:DecisionMSP} stops on the input $\gpr{X}{R}$, $w$, $h_1$, \dots, $h_k$ if and only if $w\in \gp{h_1, \ldots, h_k}$ in $G$. Furthermore, it terminates in at most $\delta_H(w)$ iterations.
The time complexity of Algorithm \ref{A:DecisionMSP}
is bounded by:
\[
	\tilde{O}\rb{\rb{\Fnorm{w} + L(H)} L(R)^{\delta_H(w)}}.
\]
\end{theorem}

\begin{proof}
Algorithm \ref{A:DecisionMSP} is a generalization of Algorithm~$\CA_{WP}$ (\cite[Algorithm $\CA$]{MSU_book:2011}).
It is straightforward to modify Theorem~\ref{Thm:wspSolverComplexity} (\cite[Theorem~16.4.3]{MSU_book:2011}) and see that
Algorithm~$\CA_{MP}$ indeed halts in at most $\delta_H(w)$ iterations.
It is easy to see that $\Gamma_i = S(\CC^{i}(\Gamma_0))$ and:
    $$|\Gamma_0| = \Fnorm{w}+\sum_{i=1}^k \Fnorm{h_i} \ \ \text{ and } \ \ |\CC^{i}(\Gamma_0)| \le |\Gamma_0| \cdot L(R)^i.$$
Since folding can be done in nearly linear time (see \cite{Touikan:2006})
Algorithm \ref{A:DecisionMSP} has the claimed time complexity.
\end{proof}


\section{Proof of the main theorems}\label{Sec:diagramsAndTrees}

In this section we prove that the challengers (Algorithms~\ref{A:esp_generator},~\ref{A:csp_generator},~\ref{A:msp_generator0},~\ref{A:msp_generator}) generically
have at most polynomial time advantage over the solvers $\CA_{WP}$, $\CA_{CP}$, and $\CA_{MP}$.

\subsection{Word search problem}
\label{sub:wspAnalysis}

Here we investigate challenges produced by Algorithm~\ref{A:esp_generator} on the fixed input
$(X;R)$ and $w=\varepsilon$.
As discussed in Section \ref{Sub:wsp_generator}, Algorithm~\ref{A:esp_generator} defines a sequence of probability measures $\{\mu_{n}\}$ on $F(X)$.
Our goal is to show that for some fixed constant $C$:
\[
\mu_{n} \Set{w' \in \FreeGroup{X}}{\delta(w') \le C \ln n } \to 1 \text{ as } n\to \infty.
\]

To construct a word $w'$ Algorithm~\ref{A:esp_generator} generates a sequence of intermediate words $\varepsilon = w_0$, $w_1$, $\dots$, $w_n$
(with $\overline{w}_n = w'$).
Each $w_i$ is obtained from $w_{i-1}$ by insertion of a word $u_i\in\ElementaryIdentities$ at the position $p_i$.
The sequence $w_0$, \dots, $w_n$ defines the sequence of van Kampen diagrams $D_0$, $D_1$, \dots, $D_n$ in a natural way as follows.
The diagram $D_0$ is the trivial diagram consisting of a single vertex.
Define a set of building blocks for van Kampen diagrams over $\gpr{X}{R}$:
\begin{gather*}
	\DiagElementaryIdentities = \Set{\text{a free edge with label $x$}}{x \in X^{\pm}} \cup
	\Set{\text{a cell with label $r$}}{r\in R}.
\end{gather*}
For $i = 1, \dots, n$ the diagram $D_{i}$
is constructed from $D_{i-1}$ by attaching $t_i \in \DiagElementaryIdentities$ labeled with $u_i$ to a vertex $\xi_i \in V(D_{i-1})$ corresponding to the position $p_i$ in $w_{i-1} = \mu(D_{i-1})$.
We call the vertex $\xi_i$ the \emph{active vertex on iteration $i$}.
The distribution $\DRandom{\ElementaryIdentities}$ of $u_i$ induces the distribution $\DRandom{\DiagElementaryIdentities}$ of $t_i$.

By construction, the diagrams have a tree-like structure and we can further define a sequence of nested trees
$T_i = (V_i, E_i)$ with $V_i = V(D_i)$ and $v_1 \to v_2 \in E_i$ if $v_1$
is an active vertex on iteration~$j$ and $v_2 \in V(D_j) \setminus V(D_{j-1})$ for some $j \le i$.
It is easy to check that $T_i$ is a tree.
This way Algorithm \ref{A:esp_generator} induces a discrete random (branching) process $\{T_i\}$ generating trees which plays a crucial role in our investigation of the properties of random identities.

It will be convenient for us to describe the process $\{T_i\}$ explicitly, avoiding words $w_i$ and diagrams $D_i$.
For each vertex $v \in V_i$ we define a number $\weight{i}(v) \in \MN$ called the \emph{weight}
of $v$ in $T_i$. The weight of the $k$th level of $T_i$ for $k \ge 0$ is:
\[
	\weight{i}_k = \sum_{d(v, v_0) = k} \weight{i}(v).	
\]
The \emph{total weight} of $T_i$ is $\Tweight{i} = \sum_{v \in V_i} \weight{i}(v)$.
The upper indices here emphasize that we consider weights for the tree $T_i$.
The next lemma shows how the sequence of trees $\{T_i\}$ evolves.
\begin{lemma}
\label{Lemma:TreeDistr}
The tree $T_0$ consists of a single vertex $v_0$ with $\weight{0}(v_0) = 1$.
For $i = 1, \dots, n$ the tree $T_{i}$ is constructed from $T_{i-1}$
by adding $\eta_i \gets \LDRandom{\ElementaryIdentities}- 1$ new children to
a random vertex $\xi_i\in V_{i - 1}$ distributed as:
	\begin{equation}
	\label{Eq:treeNodeDistr}
		\Prob{\xi_i = v} = \frac{\weight{i - 1}(v)}{\Tweight{i - 1}} \mbox{ for } v \in V_{i - 1}.
	\end{equation}
	The weight of a vertex $u \in V_{i}$ in $T_i$ satisfies the following relation:
	\begin{equation}
	\label{Eq:weightRecurrentFormula}
		\weight{i}(u) =
		\begin{cases}
			1,				  	   &\text{if $u \in V_{i} \setminus V_{i-1}$,}\\
			\weight{i - 1}(u) + \Indicator{u = \xi_i}, 	   &\text{otherwise,}
		\end{cases}
	\end{equation}
	where $\Indicator{\cdot}$ is the indicator function.
\end{lemma}
\begin{proof}
Since $t_i \gets \DRandom{\DiagElementaryIdentities}$ the number of
vertices in $t_i$ is $\eta_i \gets \LDRandom{\ElementaryIdentities}$.
Attaching $t_i$ at $\xi_i$ adds $\eta_i - 1$ new children to $\xi_i$.

Each vertex $v \in V_{i-1}$ lies on $\partial D_{i-1}$ and, therefore, corresponds to a position (possibly more than one)
of the boundary word $w_{i-1}=\mu(D_{i-1})$.
We interpret the vertex weight $\weight{i-1}(u)$ as the number of positions in $w_{i-1}=\mu(D_{i-1})$ corresponding to the vertex $u$.
Clearly, $\weight{0}(v_0) = 1$ because $V_0 = \{ v_0 \}$ and there is only one position in $w_0 = \varepsilon$.
Since each position in $\mu(D_{i-1})$ is equally likely to be chosen,
the probability of $v$ to be chosen is proportional to its weight, which gives~\eqref{Eq:treeNodeDistr}.
An attachment of a new edge or a cell increases the weight of $\xi_i$ and sets the weights for the new vertices to $1$, proving~\eqref{Eq:weightRecurrentFormula}.
\end{proof}

By $\treeHeight{T_n}$ we denote the {\em height} of the tree $T_n$ (the maximal distance from the root to a vertex).
\begin{lemma}
\label{L:treeHeightToDepth}
Let $\varepsilon = w_0$, \dots, $w_n$ be a sequence of words generated by Algorithm~\ref{A:esp_generator}, $w' = \Reduced{w}_n$, and $D_0$, \dots, $D_n$ the sequence of the corresponding diagrams.
Let $D$ be a van Kampen diagram obtained by folding the boundary of $D_n$. Then
$\delta(w') \le \delta (D) \le \treeHeight{T_n}$.
\end{lemma}

\begin{proof}
Since $\mu(D) = w'$ we have $\delta(w') \le \delta (D)$.
Folding the boundary of $D_n$ we do not fold the base vertex $v_0$ inside.
Hence,
$\delta(D) = \max_{c \in C(D)} d^\ast(c, c_{\text{out}}) \le \treeHeight{T_n}$.
\end{proof}

Denote the cumulative distribution function of $h(T_n)$, which we use throughout Section~\ref{Sec:diagramsAndTrees}, by $\CF_n(x)$:
\[
	\CF_n(x) = \Prob{h(T_n) \le x}.
\]
\begin{proposition}\label{pr:h_bound}
There exists a constant $C < e^2$ depending on $\DRandom{\ElementaryIdentities}$ such that:
$$\CF_n(C\ln n) \to 1 \text{ as } n \to \infty.$$
\end{proposition}

\begin{proof}
The random process $\{T_n\}_{n \ge 0}$ is a particular CMJ process described in Section~\ref{SubS:particularCMJ} (see Corollary~\ref{Cor:particularCMJ}).
Therefore, by Theorem~\ref{Thm:cmjTreeHeightLimite} there exists a constant $C < e^2$ satisfying:
\[
	\Prob{\lim_{n \to \infty} \frac{\treeHeight{T_n}}{\ln n} = C} = 1,
\]
which implies that for any $\varepsilon > 0$ we have:
\[
	\CF_n((C + \varepsilon) \ln n) = \Prob{h(T_n) \le (C + \varepsilon) \ln n)} \underset{n \to \infty}{\longrightarrow} 1.
\]
\end{proof}
Fix the constant $C$ defined in Proposition~\ref{pr:h_bound} for the rest of Section~\ref{Sec:diagramsAndTrees}.
\begin{theorem}
\label{Thm:IdDepthBound}
Let $\{\mu_{n}\}$ be the system of probability measures on $F(X)$
defined by Algorithm~\ref{A:esp_generator} for a fixed group presentation $\gpr{X}{R}$
and a word $w=\varepsilon$. Then:
\[
\mu_{n} \Set{w' \in \FreeGroup{X}}{\delta(w') \le C \ln n } \to 1 \text{ as } n\to \infty.
\]
\end{theorem}
\begin{proof}
By Lemma~\ref{L:treeHeightToDepth}:
\[
	\mu_{n} \Set{w' \in \FreeGroup{X}}{\delta(w') \le x} \ge \CF_n(x) \text{ for } x \in \MN.
\]
The rest follows from Proposition~\ref{pr:h_bound}.
\end{proof}
\begin{thma}{\textbf{A.}}
\emph{For any finite presentation $\gpr{X}{R}$ Algorithm $\CA_{WP}$
solves the randomized problem $(\WSP, \{\mu_n\}_{n \ge 0})$ defined by Algorithm~\ref{A:esp_generator} generically in polynomial time $\tilde{O}\rb{n^{1 + e^2 \ln L(R)}}$.
}
\end{thma}
\begin{proof}
By Theorem~\ref{Thm:wspSolverComplexity} the solver time complexity on the output $w'$ of Algorithm~\ref{A:esp_generator} is bounded by:
\[
	\tilde{O}\rb{ \Fnorm{w'} L(R)^{\delta(w')}}.
\]
We can bound $\Fnorm{w'}$ by $n \max_{r \in R} \FMnorm{r}$. By Theorem~\ref{Thm:IdDepthBound} and the fact that $C < e^2$ the generic time complexity of $\CA_{WP}$ on $w'$ is bounded by:
\[
	\tilde{O}\rb{n \max_{r \in R} \FMnorm{r} L(R)^{e^2 \ln n}} = \tilde{O}\rb{n^{1 + e^2 \ln L(R)}}.
\]
\end{proof}

\subsection{Equivalence search problem}
\label{sub:espAnalysis}

In general, Algorithm~\ref{A:esp_generator} produces words equivalent to the input $w$ in $G = \gpr{X}{R}$.
In this section we show that this general case is not harder than the case with $w=\varepsilon$ considered above
and that similar complexity bounds hold.

For a given word $w\in F(X)$ and $n\in\MN$ Algorithm~\ref{A:esp_generator} produces a sequence of words $w=w_0,\ldots,w_n$
and outputs $w' = \overline{w}_n$, which, as in the previous section,  naturally defines
a sequence of van Kampen diagrams $\{D_i\}$ with
$D_0$ is as in Figure~\ref{fi:Gw_graph} (the line segment with label $w$) and
$D_i$ is obtained from $D_{i-1}$ as described in Section \ref{sub:wspAnalysis} (see Figure~\ref{fig:esp_diagram}).
By construction, $\mu(D_i) = w \circ w_i^{-1}$.
The system of nested graphs $\{T_i\}$ corresponding to $\{w_i\}$ is defined in a similar way as the trees in Section~\ref{sub:wspAnalysis}.
Each $T_i$ is a disjoint union of $|w|+1$ rooted trees (\emph{forest}) $R_0, \dots, R_{|w|}$ with the set of roots $V_0 = V(D_0) = \left\{v_0, \dots, v_{\FMnorm{w}}\right\}$. The height of $T_i$ is:
\[
h(T_i) = \max_{0\le j\le |w|} h(R_j).
\]
\begin{figure}
    \centering
	\includegraphics[scale=0.6]{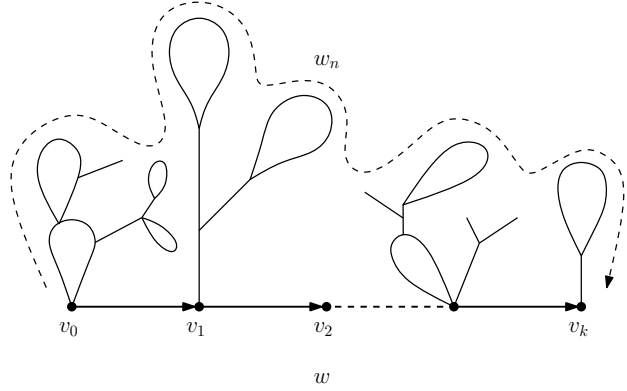}
	\caption{A diagram $D_n$.}
    \label{fig:esp_diagram}
\end{figure}

\begin{lemma}
\label{lemma:forestHeightToDepth}
Let $D$ be a van Kampen diagram obtained by folding the part of $\partial D_n$ labeled with $w_n^{-1}$ (see Figure~\ref{fig:esp_diagram}). Then $\mu(D) = w \circ w'^{-1}$ and:
\[
\delta(w \circ w'^{-1}) \le \delta(D) \le h(T_n).
\]
\end{lemma}

\begin{proof}
Similar to the proof of Lemma \ref{L:treeHeightToDepth}.
\end{proof}

It is easy to see that Formulas~\eqref{Eq:treeNodeDistr} and~\eqref{Eq:weightRecurrentFormula} hold for the sequence of forests $\{T_i\}$. Define the \emph{cumulative weight up to $k$th level of $T_i$} as $\weight{i}_{<k} = \sum_{m < k} \weight{i}_m$, and the level-weights tuple as:
\[
	\vweight{i} = \tuple{\weight{i}_0, \dots, \weight{i}_{h(T_i)} }.
\]
Set $\lp{i}_m = \frac{\weight{i}_m}{\Tweight{i}}$ which is the probability to pick a vertex from the level $m$ on the $(i+1)$-st generation step.

\begin{proposition}\label{pr:esp_h_bound}
$\Prob{h(T_n) \le x} \ge \CF_n(x) \text{ for any } x \in \MN$, where $\CF_n(x)$ is defined in Section~\ref{sub:wspAnalysis}.
\end{proposition}
\begin{proof}
Suppose $\{T'_i\}$ is a sequence of trees as in Lemma~\ref{Lemma:TreeDistr}.
Our goal is to compare $h(T_i)$ and $h(T_i')$. We use primes to distinguish characteristics of $T_i'$.
By definition:
\[
	\CF_n(x) = \Prob{h(T'_n) \le x}.
\]
To prove the proposition we use the following claim.
\begin{claim}
\label{claim:esp_partition}
Let $\PS_n$ and $\PS'_n$ be the probability spaces for $T_n$ and $T'_n$. We can partition them over the same set of indices $I_n$:
\[
	\PS_n = \sqcup_{i \in I_n} \PS_{n, i}, \ \PS'_n = \sqcup_{i \in I_n} \PS'_{n, i}
\]
so that for every $i \in I_n$ it holds that $\Prob{\PS_{n, i}} = \Prob{\PS'_{n, i}}$,
and in $\PS_{n, i}$ and $\PS'_{n, i}$:
\begin{gather}
	\text{the vectors $\vweight{n}$ and $\vweight[\prime]{n}$ are constant,}\\
	\label{Eq:total_weight_reccurence}
	\Tweight{n} = \Tweight[\prime]{n} + \FMnorm{w},\\
	\label{Eq:weight_inequality}
	\weight{n}_{<k} \ge \weight[\prime]{n}_{<k} + \FMnorm{w} \text{ for any $k \in \MN$},\\
	\label{Eq:prob_inequality}
	\lp{n}_{<k} \ge \lp[\prime]{n}_{<k}  \text{ for any $k \in \MN$}.
\end{gather}
\end{claim}
\begin{proof}
Induction on $n$. For $n = 0$ we have:
\[
	\vweight{0} = \tuple{\FMnorm{w} + 1}, \vweight[\prime]{0} = \tuple{1}
\]
for which the conditions hold. Suppose that the claim statement holds for $\PS_{n-1}$ and $\PS'_{n-1}$.
We show how for each $i \in I_{n-1}$ we can partition $S = \PS_{n - 1, i}$ and $S' = \PS'_{n - 1, i}$ in the way satisfying the claim conditions.
After that the union of these partitions for each $i \in I_{n-1}$ gives us $I_n$ and required partitions of $\PS_n$ and $\PS'_n$.

Let $\xi_n$ be the active vertex of $T_{n-1}$ on iteration $n$.
Define the random variable $K_n$ to be the level of $\xi_n$ in $T_{n-1}$, that is:
\[
	K_n = \min_{v \in V_0} d(\xi_n, v), 	
\]
where $d$ is the graph distance in $T_{n-1}$.
We partition according to the values of $\eta_n$ and $K_n$:
\begin{gather*}
	S = \bigsqcup_{d \in \MN} S_{d} \ \text{with} \ S_{d} = S \cap \{ \eta_n = d \}, \\
	S_{d} = \bigsqcup_{k \in \MN} S_{(d, k)} \ \text{with} \ S_{(d, k)} = S_{d} \cap \{ K_n = k \}.
\end{gather*}
The same way we define $\xi'_n$, $K'_n$ for $T'_{n-1}$ and the partitions of $S'$ into $\{S'_{d}\}_{d  \ge 0}$ and $\{S'_{(d, k)}\}_{d, k  \ge 0}$.
It is clear that:
\[
	\Prob{S_{d}} = \CProb{\eta_n = d}{S} \Prob{S} = \CProb{\eta'_n = d}{S'} \Prob{S'} = \Prob{S'_{d}} 	
\]
because $\Prob{S} = \Prob{S'}$ by the induction hypothesis and the conditional probabilities depend only on the distribution $\LDRandom{\ElementaryIdentities}$, which is the same for these processes. For $S_{d}$ and $S'_{d}$:
\[
	\Tweight{n} = \Tweight{n - 1} + d, \ \Tweight[\prime]{n} = \Tweight[\prime]{n - 1} + d,
\]
which are constant and together with the inductive hypothesis imply~\eqref{Eq:total_weight_reccurence}.
Also by the inductive hypothesis it holds that for any $k \in \MN$:
\begin{align*}
	\CProb{\bigsqcup_{m < k} S_{(d, m)}}{S_{d}} &= \CProb{K_n < k}{S_{d}} = \lp{n - 1}_{<k} \\
	&\ge \lp[\prime]{n - 1}_{<k} = \CProb{K'_n < k}{S'_{d}} = \CProb{\bigsqcup_{m < k} S'_{(d, m)}}{S'_{d}}. 	
\end{align*}
Therefore, we can repartition these sets:
\[
	\bigsqcup_{k \in \MN} S_{(d, k)} = \bigsqcup_{\theta \in \Theta_{d}} S_{d, \theta}, \
	\bigsqcup_{k \in \MN} S'_{(d, k)} = \bigsqcup_{\theta \in \Theta_{d}} S'_{d, \theta}
\]
in such a way that for any $\theta \in \Theta_{d}$ it holds that $\CProb{S_{d, \theta}}{S_{d}} = \CProb{S'_{d, \theta}}{S'_{d}}$ and the random variables $K_n$ and $K'_n$ are constant in $S_{d, \theta}$ and $S'_{d, \theta}$ and satisfy $K_n \le K'_n$. It follows that $\Prob{S_{d, \theta}} = \Prob{S'_{d, \theta}}$ and:
\begin{align*}
	\weight{n}_{<k} &=  \weight{n-1}_{<k} + \Indicator{k = K_n + 1} + d \cdot \Indicator{k > K_n + 1}\\
	&\ge \weight[\prime]{n-1}_{<k} + \FMnorm{w} + \Indicator{k = K_n + 1} + d \cdot \Indicator{k > K_n + 1}\\
	&\ge \weight[\prime]{n-1}_{<k} + \FMnorm{w} + \Indicator{k = K'_n + 1} + d \cdot \Indicator{k > K'_n + 1}\\
	&= \weight[\prime]{n}_{<k} + \FMnorm{w},
\end{align*}
which proves~\eqref{Eq:weight_inequality}. For the probabilities:
\[
	\lp{n}_{<k} = \frac{\weight{n}_{<k}}{\Tweight{n}} \ge \frac{\weight[\prime]{n}_{<k} + \FMnorm{w}}{\Tweight[\prime]{n} + \FMnorm{w}} \ge \frac{\weight[\prime]{n}_{<k}}{\Tweight[\prime]{n}} = \lp[\prime]{n}_{<k},
\]
which implies~\eqref{Eq:prob_inequality}.
\end{proof}

Now it follows from the claim and the law of total probability that:
\begin{align*}
	\Prob{h(T_n) \le x} = &\sum_{i \in I_n} \CProb{h(T_n) \le x}{\PS_{n, i}} \Prob{\PS_{n, i}} \\
	\ge&\sum_{i \in I_n} \CProb{h(T'_n) \le x}{\PS'_{n, i}} \Prob{\PS'_{n, i}}  = \Prob{h(T'_n) \le x},
\end{align*}
where the inequality in the middle follows from~\eqref{Eq:weight_inequality} because the length of the weight vector defines the tree height.
\end{proof}
\begin{corollary}
$\Prob{h(T_n) \le C\ln n} \to 1 \text{ as } n \to \infty.$
\end{corollary}
\begin{proof}
It follows from Propositions~\ref{pr:esp_h_bound} and~\ref{pr:h_bound}.
\end{proof}

\begin{theorem}
\label{Thm:IdDepthBound2}
For any $w \in \FreeGroup{X}$ and the corresponding system of probability measures $\{\mu_{n,w}\}$ on $\FreeGroup{X}$ defined by Algorithm~\ref{A:esp_generator}:
\[
	\mu_{n,w} \Set{w' \in \FreeGroup{X}}{\delta(w \circ w'^{-1}) \le C \ln n } \ge \CF_n(C \ln n) \to 1 \text{ as } n\to \infty.
\]
\end{theorem}
\begin{proof}
It follows from Lemma~\ref{lemma:forestHeightToDepth} and Propositions~\ref{pr:esp_h_bound} and~\ref{pr:h_bound}.
\end{proof}

\begin{thma}{\textbf{B.}}
\emph{For any finite presentation $\gpr{X}{R}$ and $w\in \FreeGroup{X}$
Algorithm $\CA_{WP}$ solves the randomized problem
$\left(\ESP(w), \{\mu_{n, w}\}_{n \ge 0}\right)$ defined by Algorithm~\ref{A:esp_generator} generically in polynomial time $\tilde{O}\rb{(\FMnorm{w} + n) n^{e^2 \ln L(R)}}$.
}
\end{thma}
\begin{proof}
By Theorem~\ref{Thm:wspSolverComplexity} the solver time complexity on the output $w'$of Algorithm~\ref{A:esp_generator} is bounded by:
\[
	\tilde{O}\rb{ \Fnorm{w'} L(R)^{\delta(w')}}.
\]
We can bound $\Fnorm{w'}$ by $\Fnorm{w} + n \max_{r \in R} \FMnorm{r}$.
By Theorem~\ref{Thm:IdDepthBound2} and the fact that $C < e^2$ the generic time complexity of $\CA_{WP}$ on $w'$ is bounded by:
\[
	\tilde{O}\rb{(\FMnorm{w} + \max_{r \in R} \FMnorm{r} n) L(R)^{e^2 \ln n}} = \tilde{O}\rb{(\FMnorm{w} + n) n^{e^2 \ln L(R)}}.
\]
\end{proof}

Note that the constant $C$ does not depend on $w$ and
the rate of convergence in Theorem \ref{Thm:IdDepthBound2} is uniformly bounded below by $\CF_n(C \ln n)$.
Therefore the following corollary holds.

\begin{corollary}
For any infinite sequence $\{u_n\}_{n \ge 0} \subset \FreeGroup{X}$:
\[
	\mu_{n,u_n} \Set{w' \in \FreeGroup{X}}{\delta(u_n \circ w'^{-1}) \le C \ln n } \to 1 \text{ as } n\to \infty.
\]
\end{corollary}

\subsection{Conjugacy search problem.}
\label{sub:cspAnalysis}
Here we investigate challenges produced by Algorithm~\ref{A:csp_generator}.

Let $w'$ be the output of Algorithm~\ref{A:csp_generator} for an input word $w$. For simplicity, assume that $w$ is cyclically reduced.
Step~$1$ of the algorithm becomes unnecessary and let $w_0$ be picked on step~$2$ and $u$ produced on step~$3$ of the algorithm.
Suppose $D$ is the corresponding to the word $w_0 \circ u^{-1}$ diagram as in Lemma~\ref{lemma:forestHeightToDepth}.
We can attach an edge with label $x$ at the end of $w_0$ to get the diagram with the boundary $w_0 \circ x \circ x^{-1} \circ u^{-1}$.
By identifying the end vertex of $w_0 \circ x$ with the base vertex $v_0$ we construct an annular diagram $A_0$ with $\mu_{\text{in}}(A_0) = w_0 \circ x$ and $\mu_{\text{out}}(A_0) = u \circ x$.
Folding the outer boundary of $A_0$, which gives $\mu_{\text{out}} = \CReduced{u \circ x}$, and picking the correct $v_{\text{out}}$, which defines a cyclic permutation of $\CReduced{u \circ x}$, we get an annular diagram $A$ with $\mu_{\text{in}}(A) = \CReduced{w}$ and $\mu_{\text{out}}(A) = w'$. The following lemma is obvious.
\begin{lemma}
\label{lemma:csp_depth_bound}
$\delta_{\conj{}}(\CReduced{w}, w') \le \delta (A) \le \delta(D)$.
\qed
\end{lemma}

\begin{theorem}
\label{Thm:ConjDepthBound}
For any $w \in \FreeGroup{X}$ and the corresponding system of probability measures $\{\nu_{n,w}\}$ on $\FreeGroup{X}$ defined by Algorithm~\ref{A:csp_generator}:
\[
	\nu_{n,w} \Set{w' \in \FreeGroup{X}}{\delta_{\conj{}}(\CReduced{w}, w') \le C \ln n } \ge \CF_n(C \ln n) \to 1 \text{ as } n\to \infty.
\]
\end{theorem}
\begin{proof}
It follows from Lemmas~\ref{lemma:csp_depth_bound} and~\ref{lemma:forestHeightToDepth} and Propositions~\ref{pr:esp_h_bound} and~\ref{pr:h_bound}.
\end{proof}

\begin{thma}{\textbf{C.}}
\emph{
For a finite presentation $\gpr{X}{R}$ and $w\in \FreeGroup{X}$ Algorithm $\CA_{CP}$ solves the randomized problem
$\left(\CSP(w), \{\nu_{n, w}\}_{n \ge 0} \right)$ defined by Algorithm~\ref{A:csp_generator} generically in polynomial time $\tilde{O}\rb{\FMnorm{\CReduced{w}}(\FMnorm{\CReduced{w}} + n) n^{2 e^2 \ln L(R)}}$.
}
\end{thma}

\begin{proof}
By Theorem~\ref{Thm:cspSolverComplexity} the time complexity of $\CA_{CP}$ on the input $(\CReduced{w}, w')$ is bounded by:
\[
	\tilde{O}\rb{ \Fnorm{\CReduced{w}} \Fnorm{w'} L(R)^{2 \delta_{\conj{}}(\CReduced{w}, w')}}.
\]
We can bound $\Fnorm{w'}$ by $\Fnorm{\CReduced{w}} + n \max_{r \in R} \FMnorm{R}$. By Theorem~\ref{Thm:ConjDepthBound} and the fact that $C < e^2$ the generic time complexity of $\CA_{CP}$ on $(\CReduced{w}, w')$ is bounded by:
\[
	\tilde{O}\rb{\FMnorm{\CReduced{w}}(\FMnorm{\CReduced{w}} + n \max_{r \in R} \FMnorm{R}) L(R)^{2e^2 \ln n}} = \tilde{O}\rb{\FMnorm{\CReduced{w}}(\FMnorm{\CReduced{w}} + n) n^{2 e^2 \ln L(R)}}.
\]
\end{proof}

Note that the constant $C$ does not depend on $w$ and the rate of convergence in Theorem \ref{Thm:ConjDepthBound} is
uniformly bounded by $\CF_n(C \ln n)$. Therefore the following corollary holds.

\begin{corollary}
For any infinite sequence $\{u_n\}_{n \ge 0} \subset \FreeGroup{X}$:
\[
	\nu_{n,u_n} \Set{w' \in \FreeGroup{X}}{\delta_{\conj{}}(\CReduced{u_n}, w') \le C \ln n } \to 1 \text{ as } n\to \infty.
\]
\end{corollary}

\subsection{Membership search problem}
\label{sub:mspAnalysis}
Here we investigate challenges produced by Algorithms~\ref{A:msp_generator0} and~\ref{A:msp_generator}.

\begin{thma}{\textbf{D.}}
\emph{
For any finite presentation $\gpr{X}{R}$ and a finite set $H \subset \FreeGroup{X}$ Algorithm $\CA_{MP}$ solves the randomized problem $\left(\MSP(H), \{\rho_{k, n, H}\}_{k, n \ge 0}\right)$ defined by Algorithm~\ref{A:msp_generator0} generically in polynomial time $\tilde{O}\rb{(k + n)n^{e^2 \ln L(R)}}$.
}
\end{thma}
\begin{proof}
By Theorem~\ref{Thm:mspSolverComplexity} the time complexity of $\CA_{MP}$ on a word $w'$ is bounded by:
\[
	\tilde{O}\rb{\rb{\Fnorm{w'}+L(H)} L(R)^{\delta_H(w')}},
\]

By Theorem~\ref{Thm:IdDepthBound2}:
\[
	\mu_{n,v} \Set{w' \in \FreeGroup{X}}{\delta(v \circ w'^{-1}) \le C \ln n } \ge \CF_n(C \ln n) \to 1 \text{ as } n\to \infty.
\]
It is clear that $\delta_H(w') \le 1 + \delta(v \circ w'^{-1})$. Since $C < e^2$ for any $k$:
\[
	\rho_{k, n, H} \Set{w' \in \FreeGroup{X}}{\delta_H(w') \le e^2 \ln n } \to 1 \text{ as } n \to \infty.
\]
We can bound $\FMnorm{w'}$ by $k \max_{h \in H} \Fnorm{h} + n \max_{r \in R} \Fnorm{r}$. Hence, the generic time complexity of $\CA_{MP}$ is bounded by:
\[
	\tilde{O}\rb{\rb{k \max_{h \in H} \Fnorm{h} + n \max_{r \in R} \Fnorm{r}+L(H)} L(R)^{e^2 \ln n}} = \tilde{O}\rb{(k + n) n^{e^2 \ln L(R)}}.
\]
\end{proof}

Let us analyze Algorithm~\ref{A:msp_generator}.
As in Section~\ref{sub:wspAnalysis} we consider a random sequence of words $\varepsilon = w_0$, $w_1$, \dots, $w_n$ generated by Algorithm~\ref{A:msp_generator} with $w' = \ovw_n$. It induces the random sequence of diagrams $D_0$, $D_1$, \dots, $D_n$, where $D_0$ is the empty diagram and $D_i$ is obtained from $D_{i-1}$ with probability $q$ by attaching a random $H$-cell to $v_0$ and with probability $1 - q$ by attaching a random $\DiagElementaryIdentities$-element. It induces a discrete random process on trees $\{ T_n \}_{n \ge 0}$ similar to the one described in Section~\ref{sub:wspAnalysis}.
It is defined by the following rules. The tree $T_0$ consists of a single vertex $v_0$ with $\weight{0}(v_0) = 1$. For $i = 1, \dots, n$ the tree $T_{i}$ is constructed from $T_{i-1}$ by adding $\eta_i - 1$ new children to a random vertex $\xi_i$.
We pick $\eta_i$ and $\xi_i$ as following:
\[
\begin{array}{llll}	
	\eta_i \gets \LDRandom{H^{\pm}}, & \xi_i \gets v_0 & \quad & \text{with probability $q$,}\\
	\eta_i \gets \LDRandom{\ElementaryIdentities}, &\xi_i \text{ satisfies~\eqref{Eq:treeNodeDistr}}  & \quad & \text{with probability $1 - q$}\\
\end{array}
\]
The weight of a vertex $u \in V_{i}$ satisfies~\eqref{Eq:weightRecurrentFormula}. Note that $\LDRandom{H^{\pm}} = \LDRandom{H}$.
\begin{lemma}
\label{lemma:mspHDepthTreeBound}
$\delta_H(w') \le \delta(D) \le h(T_n).$	
\end{lemma}
\begin{proof}
The same as in Lemma~\ref{L:treeHeightToDepth}.
\end{proof}

To get a logarithmic bound for the process above we consider another discrete branching process on trees $\{ T'_i \}_{i \ge 0}$ as follows.
The tree $T'_0$ consists of a single vertex $v'_0$ with $\weight[\prime]{0}(v'_0) = 1$.
For $i = 1, \dots, n$ the tree $T'_{i}$ is constructed from $T'_{i-1}$ by adding $\eta'_i - 1$ new children to a random vertex $\xi'_i$ satisfying~\eqref{Eq:treeNodeDistr}.
We pick $\eta'_i$ as follows:
\[
\begin{array}{ll}	
	\eta'_i \gets \LDRandom{H^{\pm}} & \text{with probability $q$,}\\
	\eta'_i \gets \LDRandom{\ElementaryIdentities}  &\text{with probability $1 - q$.}\\
\end{array}
\]
We define the weight of a vertex $u \in V'_{i}$ by~\eqref{Eq:weightRecurrentFormula}.
Proposition~\ref{pr:msp_aux_h_bound} gives a generic logarithmic bound on the height of $T'_n$ and Lemma~\ref{Lemma:mspTreeComparison} shows that $T'_n$ is probabilistically higher that $T_n$.
\begin{proposition}\label{pr:msp_aux_h_bound}
There exists a constant $C' < e^2$ depending on $\DRandom{\ElementaryIdentities}$, $\DRandom{H}$, and $q$ such that:
$$\Prob{h(T'_n) \le C'\ln n} \to 1 \text{ as } n \to \infty.$$
\end{proposition}
\begin{proof}
The same as in Proposition~\ref{pr:h_bound}.
\end{proof}
\begin{lemma}
\label{Lemma:mspTreeComparison}
For any $x \in \MN$:
\[
	\Prob{h(T_n) \le x} \ge \Prob{h(T'_n) \le x}.
\]
\end{lemma}
\begin{proof}
We use the following claim.
\begin{claim}
Let $\PS_n$ and $\PS'_n$ be the probability spaces for $T_n$ and $T'_n$. We can partition them over the same set of indices $I_n$:
\[
	\PS_n = \sqcup_{i \in I_n} \PS_{n, i}, \ \PS'_n = \sqcup_{i \in I_n} \PS'_{n, i}
\]
so that for $i \in I_n$ it holds that
$\Prob{\PS_{n, i}} = \Prob{\PS'_{n, i}}$,
and in $\PS_{n, i}$ and $\PS'_{n, i}$:
\begin{gather}
	\text{the vectors $\vweight{n}$ and $\vweight[\prime]{n}$ are constant,}\\
	\label{Eq:total_weight_equality_msp}
	\Tweight{n} = \Tweight[\prime]{n},\\
	\label{Eq:prob_inequality_msp}
	\lp{n}_{<k} \ge \lp[\prime]{n}_{<k} \text{ for any $k \in \MN$}.
\end{gather}
\end{claim}
\begin{proof}
Induction on $n$. For $n = 0$ we have:
\[
	\vweight{0} = \vweight[\prime]{0} = \tuple{1}
\]
for which the conditions hold. Suppose that it holds for $\PS_{n-1}$ and $\PS'_{n-1}$.
We show how for each $i \in I_{n-1}$ we can partition $S = \PS_{n - 1, i}$ and $S' = \PS'_{n - 1, i}$ in the way satisfying the claim conditions.
After that the union of these partitions for each $i \in I_{n-1}$ gives us $I_n$ and required partitions of $\PS_n$ and $\PS'_n$.

Let $\xi_n$ be the active vertex of $T_{n-1}$ on iteration $n$.
Define the random variable $K_n$ to be the level of $\xi_n$ in $T_{n-1}$, that is:
\[
	K_n = \min_{v \in V_0} d(\xi_n, v), 	
\]
where $d$ is the graph distance in $T_{n-1}$. In the same way we define $\xi'_n$, $K'_n$ for $T'_{n-1}$.

First, partition $S = S_{\ElementaryIdentities} \sqcup S_H$ and $S = S'_{\ElementaryIdentities} \sqcup S'_H$ corresponding to the branches where we pick $\eta_i, \eta'_i$ from $\LDRandom{\ElementaryIdentities}$ or from $\LDRandom{H}$.
By the definition of $T_n$ and $T'_n$:
\[
	\CProb{S_H}{S} = \CProb{S'_H}{S'} = q, \quad  \CProb{S_{\ElementaryIdentities}}{S} = \CProb{S'_{\ElementaryIdentities}}{S'} = 1 - q.	
\]
By the induction hypothesis it holds that $\Prob{S} = \Prob{S'}$ and, hence, $\Prob{S_H} = \Prob{S'_H}$ and $\Prob{S_{\ElementaryIdentities}} = \Prob{S'_{\ElementaryIdentities}}$.

The proofs for partitions of $S_{\ElementaryIdentities}$, $S'_{\ElementaryIdentities}$ and $S_H$, $S'_H$ are similar to the proof for $S$, $S'$ in Claim~\ref{claim:esp_partition}.
\end{proof}

Now it follows from the claim and the law of total probability that:
\begin{align*}
	\Prob{h(T_n) \le x} = &\sum_{i \in I_n} \CProb{h(T_n) \le x}{\PS_{n, i}} \Prob{\PS_{n, i}} \\
	\ge&\sum_{i \in I_n} \CProb{h(T'_n) \le x}{\PS'_{n, i}} \Prob{\PS'_{n, i}}  = \Prob{h(T'_n) \le x},
\end{align*}
where the inequality in the middle follows from~\eqref{Eq:prob_inequality_msp} because the length of the vector defines the tree height.
\end{proof}

\begin{theorem}\label{Thm:mspDepthBound}
For any finite set $H \subset \FreeGroup{X}$ and $q \in (0, 1)$ the system of probability measures $\{\rho'_{n, q, H}\}$ on $\FreeGroup{X}$ defined by Algorithm~\ref{A:msp_generator} satisfies:
\[
	\rho'_{n, q, H} \Set{w' \in \FreeGroup{X}}{\delta_H(w') \le C' \ln n } \to 1 \text{ as } n\to \infty.
\]
\end{theorem}
\begin{proof}
By Lemma~\ref{lemma:mspHDepthTreeBound}:
\[
	\rho'_{n, q, H} \Set{w' \in \FreeGroup{X}}{\delta_H(w') \le x} \ge \Prob{h(T_n) \le x}.	
\]
The rest follows from Proposition~\ref{pr:msp_aux_h_bound} and Lemma~\ref{Lemma:mspTreeComparison}.
\end{proof}

\begin{thma}{\textbf{E.}}
\emph{
For any finite presentation $\gpr{X}{R}$ and a finite set $H \subset \FreeGroup{X}$ Algorithm $\CA_{MP}$ solves the randomized problem $\left(\MSP(H), \{\rho'_{n, q, H}\}_{n \ge 0}\right)$ defined by Algorithm~\ref{A:msp_generator} generically in polynomial time $\tilde{O}\rb{n^{1 + e^2 \ln L(R)}}$.
}
\end{thma}

\begin{proof}
By Theorem~\ref{Thm:mspSolverComplexity} the time complexity of $\CA_{MP}$ on a word $w'$ produced by Algorithm~\ref{A:msp_generator} is bounded by:
\[
	\tilde{O}\rb{\rb{\Fnorm{w'}+L(H)} L(R)^{\delta_H(w')}},
\]
We can bound $\Fnorm{w'}$ by $n \max_{x \in R \cup H} \FMnorm{x}$. By Theorem~\ref{Thm:mspDepthBound} and the fact that $C' < e^2$ the generic time complexity of $\CA_{MP}$ on $w'$ is bounded by:
\[
	\tilde{O}\rb{\rb{n \max_{x \in R \cup H} \FMnorm{x}+L(H)} L(R)^{e^2 \ln n}} = \tilde{O}\rb{n^{1 + e^2 \ln L(R)}}.
\]
\end{proof}


\appendix
\section{Crump-Mode-Jagers process}
In this section we show that the (discrete) branching process $\{T_i\}$ of Section~\ref{sub:wspAnalysis} (and similar processes in Sections~\ref{sub:espAnalysis}, \ref{sub:cspAnalysis} and \ref{sub:mspAnalysis}) can be analyzed as a particular CMJ process.
In Section~\ref{SubS:CMJDefinition} we define a general Crump-Mode-Jagers (CMJ) process,
which is a continuous time branching process and
discuss some general techniques for studying properties of random trees constructed by CMJ processes.
In Section~\ref{SubS:particularCMJ} we show that the random trees process in Section~\ref{sub:wspAnalysis} can be viewed as a particular CMJ-processes.
In Section~\ref{sub:particularCMJproperties} we show an asymptotic bound for heights of random trees.

\subsection{Crump-Mode-Jagers process}
\label{SubS:CMJDefinition}
CMJ processes were introduced as a model of population growth (\cite{Crump-Mode:1968}). They also found
applications in a study of random trees (see e.g., \cite{Devroye:1986, Devroye:1987, Mahmoud:1994, Pittel:1984, Pittel:1985, Pittel:1994,Biggins-Grey:1997} with an overview in~\cite{Devroye:1998}).
Formally, a CMJ process is a continuous time age-dependent branching process (see~\cite{Crump-Mode:1968} for the original definition)
defined by a pair $(l,z)$, where $l$ is a positive real-valued random variable called the {\em lifespan} of a vertex and $z(t)$  is a point process
(a positive, nondecreasing, right continuous, integer-valued random process) defining the number of offsprings for each vertex.
Note that $l$ and $z$ are not necessarily independent of each other.

A CMJ process associated with a pair $(l,z)$ is defined as a random tree $\{R(t)\}_{t \in \MR_{\ge 0}}$
growing incrementally over continuous time satisfying the following properties. For every $t \in \MR_{\ge 0}$:
    $$R(t) = (V(t), E(t))$$
is a rooted tree and
$$V(t) = V'(t) \sqcup V^{\dagger}(t),$$
where $V'(t)$ is the set of {\em active} (alive) vertices and $V^{\dagger}(t)$
the set of {\em nonactive} (dead) vertices.
Each vertex $v$ gets its own independent copy $(l_v, z_v)$ of $(l, z)$ with the same joint distribution. Initially:
$$R(0) = (\{v_0\}, \emptyset).$$
For $s < t$ the tree $R(s)$ is a (rooted) subtree of $R(t)$ such that:
$$V(s)\subseteq V(t) \mbox{ and } V^{\dagger}(s)\subseteq V^{\dagger}(t).$$
In particular, $v_0$ is the root of $R(t)$ for every $t \in \MR_{\ge 0}$.
Active vertices independently produce children according to the process $z$, and nonactive ones do not.
The \emph{birth time} of $v \in V(t)$ is:
	\[
		b(v) = \min_s \left\{v \in V(s)\right\}.
	\]
By definition $b_{v_0} = 0$. Each vertex is created active, in particular, $v_0\in V'(0)$.
The {\em lifespan} of $v \in V(t)$ is given by $l_v$ and it holds:
	\[
		\min_s \left\{v \in V^{\dagger}(s)\right\} = b(v) + l_v.
	\]
For $(l, z)$ define the following point process:
\[
	y(t) = \begin{cases}
		z(t), &\text{if } t < l;\\
		z(l), &\text{if } t \ge l.
	\end{cases}
\]
Denote the distribution of $y(t)$ by $Y(t)$. The {\em number of children} of each vertex $v \in V(t)$ is $c_v(t) = y_v(t - b(v))$ for $t \ge b(v)$, where $y_v$ is the copy of $y$ corresponding to $l_v$ and $z_v$. For $t < b(v)$ set $c_v(t) = 0$.
Note that new children can appear in batches of size more than $1$ depending on the point process $z(t)$.

Denote by $\xi(t)$ the number of active vertices $|V'(t)|$ at time $t$.
\begin{remark}
In population growth models the Crump-Mode-Jagers (CMJ) process is defined as $\xi(t)$ (the population size at time $t$).
\end{remark}

\begin{remark}
The model allows the lifespan $l$ to be $\infty$ in which case $V^\dagger(T)=\emptyset$ for every $t$.
\end{remark}

Some useful characteristics of the random process $\{R(t)\}_{t \in \MR_{\ge 0}}$ are listed below.
\begin{itemize}
	\item $t_n$ is the time at which the $n$th batch of vertices appears in $R(t)$.
	\item $R_n = R(t)$ for $t_n \le t < t_{n+1}$.
	\item $h_n$ is the height of $R_n$.
	\item $\xi_n = \xi (t_n)$ is the number of active vertices of $R_n$.
	\item $B_k = \min_{n \in \MN} \Set{t_n}{h_n = k}$ is the moment of time at which tree becomes of height $k$.
\end{itemize}
By definition $t_0 = B_0 = h_0 = 0$.

The {\em intensity measure} $M(t)$ of the point process $y(t)$ is defined by $M(t) = \ME{y(t)}$. Its Laplace transform is the function:
\begin{equation}\label{eq:M_Laplace}
	m(\theta) = \int e^{-\theta s} dM(s) = \ME{\int e^{-\theta s} dY(s)}.
\end{equation}
The event of an {\em ultimate survival} is defined by:
\[
	\CS = \{V'(t)\ne \emptyset \mid \ \forall t\ge 0 \}.
\]
To ensure a positive probability of $\CS$ we need the process to be \emph{supercritical}, that is, $\lim_{t \to \infty} M(t) > 1$.
Equivalently the process is supercritical if $m(0) > 1$.
Under some mild conditions, the \emph{Malthusian parameter} of the CMJ process can be defined by:
\begin{equation}
\label{Eq:alpha}
	\alpha = \inf\Set{\theta}{m(\theta) \le 1},
\end{equation}
and for a supercritical process it is true that $\alpha > 0$. If $m(0) = \infty$, then $m(\theta_0)<\infty$ for some $\theta_0 > 0$.
In fact:
\begin{equation}
\label{Eq:supercriticalProcessCondition}
	1 < m(\theta_0) < \infty.	
\end{equation}
Fix $\theta_0$ satisfying~\eqref{Eq:supercriticalProcessCondition} and define a function:
\[
	\mu(a) = \inf \Set{e^{\theta a} m(\theta)}{\theta \ge \theta_0}.	
\]
It is an increasing function of $a \ge 0$, and since $m(\theta) \to 0$ as $\theta \to \infty$ (by monotone convergence) it holds that $\mu(a) \to 0$ as $a \to 0$.
Hence $\mu(a) < 1$ for small values of $a > 0$ and we can define a constant:
\begin{equation}
\label{Eq:gamma}
	\gamma = \sup\Set{a}{\mu(a) < 1}.	
\end{equation}
These notations are necessary for the next theorem which gives us an asymptotic relation between $B_n$ and $n$.
\begin{theorem}[Kingman, \cite{Kingman:1975}]
\label{Thm:Kingman}
If $\theta_0>0$ satisfies~\eqref{Eq:supercriticalProcessCondition} then:
\[
	\lim_{n \to \infty} \frac{B_n}{n} = \gamma
\]
holds almost surely on $\CS$.
\qed
\end{theorem}

We say that the process $\xi(t)$ is \emph{non-lattice} if the intensity measure $M(t)$ is non-lattice,
that is, it is not supported by any lattice $\{0, c, 2c, \dots \}$ with $c > 0$.
The next theorem is a simplified version of Theorem~$2$ in~\cite{Biggins:1995}.
\begin{theorem}[Biggins, \cite{Biggins:1995}]
\label{Thm:Biggins}
Let $\xi(t)$ be a supercritical non-lattice CMJ process with Malthusian parameter $\alpha$. Then:
\[
	\lim_{t \to \infty} \frac{\ln \xi(t)}{t} = \alpha
\]
holds almost surely on $\CS$.
\qed
\end{theorem}

\begin{remark}
In the general version of the theorem there is a random characteristic $\chi$, which modifies the counting of active vertices, so that:
\[
	\xi^{\chi}(t) = \sum_{v \in V'(t)} \chi_v(t - b(v)),
\]
where each vertex is assigned its own copy of the characteristic. This characteristic is subject to certain conditions, but in our case $\chi(t) = 1$ for all $t \ge 0$ and these conditions hold. In addition, this theorem is a special case of a theorem for a spatial CMJ process (see~\cite{Biggins:1995} for details).
\end{remark}

Condition~\eqref{Eq:supercriticalProcessCondition} is sufficient for the process to be supercritical and so we assume that it holds when we apply Theorems \ref{Thm:Kingman} and \ref{Thm:Biggins}.
We also assume that $t_n \to \infty$ almost surely on $\CS$ as $n \to \infty$, which allows us to replace $t$ with $t_n$ in Theorem~\ref{Thm:Biggins} to obtain:
\begin{equation}
\label{Eq:timeSizeRatio}
	\lim_{n \to \infty} \frac{t_n}{\ln \xi_n} = \frac{1}{\alpha} \quad \text{a.s. on $\CS$}.
\end{equation}
Clearly $B_{h_n} \le t_n < B_{h_{n+1}}$ and hence (provided $t_n \to \infty$):
\[
	\frac{B_{h_n}}{h_n} \le \frac{t_n}{h_n} < \frac{B_{h_n + 1}}{h_n},
\]
which together with Theorem~\ref{Thm:Kingman} implies that:
\begin{equation}\label{eq:timeHeightRatio}
\lim_{n \to \infty} \frac{t_n}{h_n} = \gamma \quad \text{a.s. on $\CS$}.
\end{equation}
Equalities~\eqref{Eq:timeSizeRatio} and~\eqref{eq:timeHeightRatio} imply the following proposition,
which was proven for some particular instances of CMJ processes and also was proven in general in the works cited in the beginning of the section (see~\cite{Devroye:1998} for a general overview).

\begin{proposition}\label{Prop:mainCMJ}
Let $\xi(t)$ be a non-lattice CMJ process
for which there exists $\theta_0>0$ satisfying~\eqref{Eq:supercriticalProcessCondition}
and $t_n \to \infty$ as $n \to \infty$. Then:
\[
	\lim_{n \to \infty} \frac{h_n}{\ln \xi_n} = \frac{1}{\alpha \gamma} \quad \text{a.s. on $\CS$,}
\]
where $\alpha$ and $\gamma$ are defined in~\eqref{Eq:alpha} and~\eqref{Eq:gamma}.
\qed
\end{proposition}

\subsection{$\{T_i\}$ as a CMJ-process}
\label{SubS:particularCMJ}

Let $\CM$ be a distribution on $\MN$.
For $\CM$ we can define a step process $z(t)=z_\CM(t)$ as follows.
Initially, $z(0)=0$.
The sequence $0=s_0,s_1,s_2,\ldots \in\MR_+$ of steps of $z$ and their size is defined by:
\begin{itemize}
\item
$\tau_i = s_i-s_{i-1}$ is an independent random variable distributed as $\ExpD{i}$, an exponential random variable with parameter $i$;
\item
$\nu_i = z(s_i)-z(s_{i-1})$ is an independent random variable distributed as $\CM$.
\end{itemize}
Formally, $z$ can be defined as a weighted sum of indicators:
\begin{gather*}
z_\CM(t) = \sum_{k = 1}^\infty \nu_i \Indicator{t \ge \sum_{j = 1}^k \tau_j} \\
	=
	\begin{cases}
		0,  &\text{if $t < \tau_1$;}\\
		\nu_1 + \dots + \nu_i,  &\text{if $\tau_1 + \dots + \tau_i \le t < \tau_1 + \dots + \tau_i + \tau_{i + 1}$.}
	\end{cases}
\end{gather*}

Further, we define a CMJ-process $\{R(t)\}_{t\in\MR_{\ge0}}$ with an infinite lifespan $l$
and the offspring-size function $z_\CM(t)$.
Our goal is to show that the process $\{T_i\}$ defined in Section \ref{sub:wspAnalysis}
and the discrete process $\{R_i\}$ are the same if $\CM=\LDRandom{\ElementaryIdentities} - 1$.

For a vertex $v \in V(R_i)$ let $\{\nu_{i}(v), \tau_{i}(v)\}_{i \ge 0}$ be its copies of the random variables defining $z$. Define:
\begin{align*}
	\lambda^{(i)}(v) &= 1 + \max_k \left\{ b(v)+\sum_{j = 1}^k \tau_{j}(v) \le t_i \right\}, \\
	\omega^{(i)}(v) &=  \sum_{j = 1}^{\lambda^{(i)}(v)} \tau_{j}(v) - t_i.
\end{align*}
The maximum in the expression for $\lambda^{(i)}(v)$ is the number of times $v$
produced children up to the moment $t_i$ and $\omega^{(i)}(v)$ is the time from $t_i$ to the next moment $v$ produces children.
It is easy to see that:
\begin{itemize}
\item
$\lambda^{(0)}(v_0) = 1$.
\item
$\lambda^{(i)}(v_0) = 1$ for each $v \in V(R_{i}) \setminus V(R_{i-1})$.
\item
For each $v \in V(R_{i-1})$:
\[
	\lambda^{(i)}(v) = \lambda^{(i-1)}(v) + \Indicator{\text{$v$ produced vertices at time $t_i$}}.	
\]
\end{itemize}
Because of the memoryless property of the exponential distribution, $\omega^{(i)}(v) \gets \ExpD{\lambda^{(i)}(v)}$, the same way as $\tau_{\lambda^{(i)}(v)}(v)$.
Since the values $\lambda^{(i)}(v)$ and $\weight{i}(v)$ have the same recurrence relations and initial conditions we get:
\[
	\lambda^{(i)}(v) = \weight{i}(v). 	
\]
The probability of each particular $v' \in V(R_i)$ to produce the next batch of children constituting $V(R_{i+1}) \setminus V(R_{i})$ is:
\[
	\Prob{\omega^{(i)}(v') = \min_{v \in V_i} \omega^{(i)}(v)} = \frac{\lambda^{(i)}(v')}{\sum_{v \in V(R_i)} \lambda^{(i)}(v)} = \frac{\weight{i}(v)}{\Tweight{i}},
\]
where the first equality follows from the properties of exponential random variables.

We summarize the properties of the trees $R_i$ in the following lemma.
\begin{lemma}\label{le:CMJtreeDistr}
The tree $R_0$ consists of a single vertex $v_0$ with $\weight{0}(v_0) = 1$.
For $i = 1, \dots, n$ the tree $R_{i}$ is constructed from $R_{i-1}$ by adding $\nu_i$ new children to a random vertex $\xi_i$, where $\nu_i \gets \mathcal{M}$ and for a vertex $v \in V_{i - 1}$ the probability $\Prob{\xi_i = v}$ is defined by~\eqref{Eq:treeNodeDistr} and for a vertex $u \in V_{i}$ the weight is defined by~\eqref{Eq:weightRecurrentFormula}.
\end{lemma}

\begin{corollary}
\label{Cor:particularCMJ}
If $\CM = \LDRandom{\ElementaryIdentities} - 1$, then the random processes $\{T_i\}_{i \ge 0}$ and $\{R_i\}_{i \ge 0}$ are the same.
\end{corollary}
\begin{proof}
It follows from Lemmas \ref{Lemma:TreeDistr} and \ref{le:CMJtreeDistr}.
\end{proof}

\subsection{Properties of $\{R_i\}$}
\label{sub:particularCMJproperties}

It is clear that the process $\{R_i\}$ is non-lattice and the event of the ultimate survival $\CS$ is the whole probability space.
Below we show that the rest of the assumptions of Proposition \ref{Prop:mainCMJ} hold for $\{R_i\}$ assuming that $\supp \CM$ is finite.

\begin{lemma}\label{Lemma:t_limit}
For the process $\{R_i\}$ defined in Section~\ref{SubS:particularCMJ} with finite $\supp \CM$:
\[
	t_n \underset{n \to \infty}{\longrightarrow} \infty \ \text{a.s.}
\]
\end{lemma}

\begin{proof}
Clearly, $t_n = t_0 + \sum_{i = 1}^{n} (t_{i} - t_{i-1})$.
Since $\omega^{(i-1)}(v)$ is distributed as $\ExpD{\weight{i-1}(v)}$,
the interbirth time $t_{i} - t_{i-1}$ is distributed as:
$$
\min_{v\in V_{i-1}} \ExpD{\weight{i-1}(v)} = \ExpD{\Tweight{i-1}}.
$$
It is clear that:
\[
	\Tweight{i} = \Tweight{i - 1} + \nu_{i} + 1 = \Tweight{0} + \sum_{k = 1}^{i} (\nu_k + 1) = 1 + i + \sum_{k = 1}^{i} \nu_k,
\]
which implies:
\[
	1 + 2 i \le \Tweight{i} \le 1 + i M, \text{ where } M = \max_{m \in \supp(\mathcal{M})} m.
\]
Hence $t_n$ is stochastically larger than $S_n = \sum_{i = 0}^{n-1} e_i$, where $e_i \gets \ExpD{1 + i M}$,
which clearly (by Chebyshev's inequality) satisfies the property $	 S_n \underset{n \longrightarrow \infty}{\rightarrow}\infty \quad \text{a.s.}$
\end{proof}

\begin{lemma} \label{Lemma:cmjInstanceParameters}
For the process $\{R_i\}$ defined in Section~\ref{SubS:particularCMJ} with finite $\supp \CM$:
\begin{gather*}
	m(\theta) = \frac{\AvgM}{\theta - 1} \ \text{for $\theta > 1$,} \\
	\alpha = \AvgM + 1, \, \mu(a) = a e^{a + 1} \AvgM,
\end{gather*}
where $\gamma$ is the unique root of $a e^{a + 1} = \frac{1}{\AvgM}$.
\end{lemma}
\begin{proof}
Recall that $m(\theta)$ is the Laplace transform of $M(t) = \ME{y(t)}$:
\[
	m(\theta) = \ME{\int_0^\infty e^{-\theta s} dY(ds)}.
\]
Note that for an indicator function $I(t) = \Indicator{t \ge c}$, where $c$ is a constant, it holds that $\int_0^\infty e^{-\theta s} dI(s) = e^{-\theta c}$, and it holds that:
\begin{equation*}
	\ME{\nu_i} = \AvgM, \
	\ME{e^{-\theta \tau_j}} = \int_{0}^{\infty} e^{-\theta x} j e^{-jx} dx = \frac{j}{j + \theta}.
\end{equation*}
We use these facts in the following derivation:
\begin{align*}
	m(\theta)
	&=
	\ME{\sum_{i = 1}^\infty \nu_i e^{-\theta \left(\sum_{j = 1}^i \tau_j\right)}}
	=
	\sum_{i = 1}^{\infty} \ME{\nu_i \prod_{j = 1}^{i} e^{-\theta \tau_j}}\\
	&\overset{\text{i.r.v.}}{=\joinrel=}
	\sum_{i = 1}^{\infty} \ME{\nu_i} \prod_{j = 1}^{i} \ME{e^{-\theta \tau_j}}
	=
	\AvgM \sum_{i = 1}^{\infty} \prod_{j = 1}^{i} \frac{j}{j + \theta}.	
\end{align*}
In~\cite{Pittel:1994} and, more directly, in~\cite{Biggins-Grey:1997} (on page $341$ for the linear recursive tree with $b = 1$) it is shown that $\sum_{i = 1}^{\infty} \prod_{j = 1}^{i} \frac{j}{j + \theta} = \frac{1}{\theta - 1}$ for $\theta > 1$, hence:
\begin{gather*}
	m(\theta) = \frac{\AvgM}{\theta - 1} \ \text{for $\theta > 1$}, \\
	\alpha = \inf \Set{\theta}{\frac{\AvgM}{\theta - 1} < 1} = \AvgM + 1, \\
	\mu(a) = \inf \Set{e^{\theta a}\frac{\AvgM}{\theta - 1}}{\theta > 1} = \AvgM a e^{a + 1}.
\end{gather*}
The function $\mu(a)$ is positive and increasing for positive values of $a$. Therefore, $\gamma = \sup\Set{a}{\mu(a) < 1}$ is the unique root of $a e^{a + 1} = \frac{1}{\AvgM}$.
\end{proof}

\begin{theorem}\label{Thm:cmjTreeHeightLimite}
For the process $\{R_i\}$ defined in Section~\ref{SubS:particularCMJ} with finite $\supp \CM$:
\begin{equation}
\label{Eq:cmjTreeHeightLimit}
	\frac{h(R_n)}{\ln n} \to \frac{1}{\alpha \gamma} \quad \text{a.s. as $n \to \infty$},
\end{equation}
where $\alpha$ and $\gamma$ are defined in Lemma~\ref{Lemma:cmjInstanceParameters}.
Moreover, $\frac{1}{\alpha \gamma} < e^2$.
\end{theorem}

\begin{proof}
It follows from Lemmas~\ref{Lemma:t_limit} and~\ref{Lemma:cmjInstanceParameters}
that Proposition~\ref{Prop:mainCMJ} is applicable to $\{R_n\}$ and, hence:
\[
	\frac{h(R_n)}{\ln \xi_n} \to \frac{1}{\alpha \gamma} \quad \text{a.s. as $n \to \infty$}.
\]
Clearly, $1 + n \le \xi_n \le 1 + n M < (n+1) M$,
where $M = \max_{m \in \supp(\mathcal{M})} m$.
Therefore:
\[
	\ln (n + 1) \le \ln \xi_n < \ln (n + 1) + \ln M
\]
and
\[
	\frac{\ln \xi_n}{h(R_n)} -  \frac{\ln M}{h(R_n)} < \frac{\ln (n + 1)}{h(R_n)} \le \frac{\ln \xi_n}{h(R_n)}.
\]
Taking the limit as $n \to \infty$ proves \eqref{Eq:cmjTreeHeightLimit}.

Finally, by Lemma~\ref{Lemma:cmjInstanceParameters}, $\alpha = \AvgM + 1$ and
$\gamma$ is the unique root of $\gamma e^{\gamma + 1} = \frac{1}{\AvgM}$.
Clearly $\gamma < 1$ and $\gamma = \frac{1}{e^{\gamma + 1} \AvgM} > \frac{1}{e^2 \AvgM}$, which implies:
\[
	\alpha \gamma > \frac{\AvgM + 1}{e^2 \AvgM} > e^{-2}.
\]
\end{proof}

\providecommand{\bysame}{\leavevmode\hbox to3em{\hrulefill}\thinspace}
\providecommand{\MR}{\relax\ifhmode\unskip\space\fi MR }
\providecommand{\MRhref}[2]{%
  \href{http://www.ams.org/mathscinet-getitem?mr=#1}{#2}
}
\providecommand{\href}[2]{#2}

\end{document}